\documentclass[oneside,english]{amsart}

\usepackage[T1]{fontenc}
\usepackage{prettyref}
\usepackage{bm}
\usepackage[colorlinks=true]{hyperref}

\makeatletter
\theoremstyle{plain}
\newcommand{\nm@name}{Undefined Theorem Name}
\newtheorem*{nm@thm}{\nm@name}
\newenvironment{namedthm}[1][Undefined Theorem Name]{
  \renewcommand\nm@name{#1}\begin{nm@thm}}{\end{nm@thm}}
\makeatother

\numberwithin{equation}{section}
\numberwithin{figure}{section}

\theoremstyle{plain}
\newtheorem{thm}{Theorem}[section]
\newtheorem{question}[thm]{Question}
\newtheorem{cor}[thm]{Corollary}
\newtheorem*{cor*}{Corollary}
\newtheorem{prop}[thm]{Proposition}
\newtheorem{lem}[thm]{Lemma}
\theoremstyle{remark}
\newtheorem{rem}[thm]{Remark}

\newtheorem*{warn}{Warning}
\newtheorem*{claim*}{Claim}
\theoremstyle{definition}
\newtheorem{defn}[thm]{Definition}
\newtheorem{note}[thm]{Note}

\newrefformat{thm}{{Theorem~\ref{#1}}}
\newrefformat{def}{{Definition~\ref{#1}}}
\newrefformat{prop}{{Proposition~\ref{#1}}}
\newrefformat{lem}{{Lemma~\ref{#1}}}
\newrefformat{cor}{{Corollary~\ref{#1}}}
\newrefformat{sec}{{Section~\ref{#1}}}
\newrefformat{sub}{{\S{\ref{#1}}}}
\newrefformat{fig}{{Figure~\ref{#1}}}
\newrefformat{note}{{Note~\ref{#1}}}

\newcommand{\A}{\mathbb{A}}
\newcommand{\C}{\mathbb{C}}
\newcommand{\N}{\mathbb{N}}
\newcommand{\Pb}{\mathbb{Pb}}
\newcommand{\Q}{\mathbb{Q}}
\newcommand{\R}{\mathbb{R}}
\newcommand{\Z}{\mathbb{Z}}
\newcommand{\Oc}{\mathcal{O}}
\newcommand{\Gm}{\mathbb{G}_{\rm m}}
\newcommand{\bt}{\mathbf{t}}
\newcommand{\bk}{\mathbf{k}}
\newcommand{\bn}{\mathbf{n}}

\usepackage[textwidth=15cm]{geometry}

\begin{document}

\title[On fewnomials]{On fewnomials, integral points and a toric
  version of Bertini's theorem}

\author{Clemens Fuchs}
\thanks{C.F.\ was supported by FWF (Austrian Science Fund) grant No.\
  P24574.}
\address[C.\ Fuchs]{University of Salzburg, Hellbrunnerstr.\ 34/I,
  A-5020 Salzburg}
\email{clemens.fuchs@sbg.ac.at}

\author{Vincenzo Mantova}
\thanks{V.M.\ was supported by the Italian FIRB 2010 ``New advances in
  the Model Theory of exponentiation''.}
\address[V.\ Mantova]{University of Camerino, Via Madonna delle
  Carceri 9, IT-62032 Camerino}
\curraddr{University of Leeds, LS2 9JT Leeds, UK}
\email{v.l.mantova@leeds.ac.uk}

\author[U.\ Zannier]{Umberto Zannier}
\thanks{The authors were also supported by the ERC-AdG 267273
  ``Diophantine Problems''.}
\address[U.\ Zannier]{Scuola Normale Superiore, Piazza dei Cavalieri
  7, IT-56126 Pisa}
\email{umberto.zannier@sns.it}

\begin{abstract}
  An old conjecture of Erd\H{o}s and R\'enyi, proved by Schinzel,
  predicted a bound for the number of terms of a polynomial
  $g(x)\in\C[x]$ when its square $g(x)^2$ has a given number of
  terms. Further conjectures and results arose, but some fundamental
  questions remained open.

  In this paper, with methods which appear to be new, we achieve a
  final result in this direction for completely general algebraic
  equations $f(x,g(x))=0$, where $f(x,y)$ is monic of arbitrary degree
  in $y$, and has boundedly many terms in $x$: we prove that the
  number of terms of such a $g(x)$ is necessarily bounded. This
  includes the previous results as extremely special cases.

  We shall interpret polynomials with boundedly many terms as the
  restrictions to 1-parameter subgroups or cosets of regular functions
  of bounded degree on a given torus $\Gm^l$. Such a viewpoint shall
  lead to some best-possible corollaries in the context of finite
  covers of $\Gm^l$, concerning the structure of their integral points
  over function fields (in the spirit of conjectures of Vojta) and a
  Bertini-type irreducibility theorem above algebraic multiplicative
  cosets.  A further natural reading occurs in non-standard
  arithmetic, where our result translates into an algebraic and
  integral-closedness statement inside the ring of non-standard
  polynomials.
\end{abstract}

\date{January 10th, 2017}

\maketitle

\section{Introduction}
\label{sec:intro}

This paper is concerned with algebraic equations involving
\textit{fewnomials}, also sometimes called \textit{sparse}, or
\textit{lacunary} polynomials.  By this we mean that the number of
\textit{terms} is thought as being fixed, or bounded, whereas the
degrees of these terms may vary, and similarly for the coefficients
(though they are sometimes supposed to be fixed as well).

This context traces back to several different viewpoints and
motivations. For instance, there are issues of reducibility (as in the
well-known old theory of \textsc{A.\ Capelli} for binomials, and in
more recent investigations for $k$-nomials, e.g.\ by \textsc{A.\
  Schinzel} \cite{Schinzel2000}). Sparse polynomials also occur when
thinking of \textit{complexity} in writing down an algebraic
expression; see for instance \textsc{J.\ Davenport}'s paper
\cite{Davenport2009} (which also mentions issues related to the ones
considered below). In turn, low complexity affects important
geometrical or topological aspects (as in \textsc{A.\ Khovanskii's}
theory \cite{Khovanski1991}).

One perspective and series of relevant questions appeared when
\textsc{P.\ Erd\H{o}s} and \textsc{A.\ R\'enyi} raised independently
the following attractive conjecture: \textit{Suppose that $g(x)$ is a
  (complex) polynomial such that $g(x)^{2}$ has at most $l$
  terms. Then the number of terms of $g(x)$ is bounded dependently
  only on $l$} \cite{Erdos1949}. It turned out that this problem was
not innocuous as it might appear; indeed, for infinitely many $l$ the
number of terms of $g(x)$ may be much larger than that of $g(x)^{2}$,
in fact $>l^{c}$ for a $c>1$, as was pointed out by Erd\H{o}s himself
\cite{Erdos1949,Schinzel2000}.

The conjecture was proved by Schinzel \cite{Schinzel1987}, actually
for $g(x)^{d}$ for any given $d > 0$. Schinzel also extended the
conjecture to compositions $p(g(x))$ for any given
$p\in\C[x]\setminus\C$, which could not be dealt with by his
methods. In turn, this was settled in \cite{Zannier2008}.

\subsection{Main results}
\label{sub:main}

One of the main purposes of the present paper is to achieve a `final'
result in the said direction, by treating general algebraic equations
$f(x,g(x))=0$, assuming that $f(x,y)\in\C[x,y]$ is a `fewnomial' in
$x$ and has arbitrary degree in $y$; we then seek a bound for the
number of terms of $g(x)\in\C[x]$.  We shall indeed prove that such a
bound exists and that it is actually uniform in the coefficients of
$f$, recovering the above mentioned conclusions related to the
Erd\H{o}s-R\'enyi conjecture (in sharper form) as very special
cases. For instance, we prove the following.

\begin{thm}\label{thm:main-simple}
  Let $f(x,y)\in\C[x,y]$ have $l$ terms in $x$ and be monic of degree
  $d > 0$ in $y$. If $g(x)\in\C[x]$ satisfies $f(x,g(x))=0$, then
  $g(x)$ has at most $B=B(d,l)$ terms.
\end{thm}

The Erd\H{o}s-R\'enyi conjecture is re-obtained on taking
$f(x,y)=y^2-h(x)$ and also Schinzel's subsequent conjecture with
$f(x,y)=p(y)-h(x)$ (moreover uniformly in the coefficients of $p$).

Results of this type are strongly related to other (apparently far)
issues of arithmetic and geometric nature, as we now
illustrate. First, we remark that a convenient point of view, adopted
here, is to think of a (Laurent) fewnomial as \textit{the restriction
  of a given regular function on a torus $\Gm^{l}$ to a $1$-parameter
  subgroup or coset.} Indeed, a regular function on $\Gm^{l}$ is just
a Laurent polynomial $f(t_{1},\ldots,t_{l})$, whereas any connected
$1$-parameter subgroup (resp.\ coset) may be parametrized as
$t_{1}=x^{m_{1}},\ldots,t_{l}=x^{m_{l}}$ (resp.\
$t_{1}=c_{1}x^{m_{1}},\ldots,t_{l}=c_{l}x^{m_{l}}$) for integers
$m_{1},\ldots,m_{l}$ (resp.\ and nonzero constants
$c_{1},\ldots,c_{l}$). Hence, by substitution inside $f$, we obtain a
Laurent polynomial in $x$ whose number of terms is bounded
independently of the subgroup or coset.  \footnote{Naturally, a
  similar interpretation holds for multivariate fewnomials; however,
  the issues may be usually reduced to the basic case of a single
  variable by substitution.}

In this view, the above theorem can be rephrased in the following
equivalent form.

\begin{thm}\label{thm:main-poly}
  If $f\in\C[t_{1},\dots,t_{l},y] \setminus \C$ is monic in $y$ and of
  degree at most $d$ in each variable, if $n_{1},\dots,n_{l}$ are
  natural numbers, and if $g(x)\in\C[x]$ satisfies
  \begin{equation}
    f(x^{n_{1}},\dots,x^{n_{l}},g(x))=0,\label{eq:main-poly}
  \end{equation}
  then $g(x)$ has at most $B_{1}=B_{1}(d,l)$ terms.
\end{thm}

The numbers $B, B_{1}$ are actually effective, although we skip the
details of such calculation. This leads, as we shall see, to a
complete algorithmic description of all the possible solutions
$g(x)$. Note that moreover the bound is independent of the
coefficients of $f$, so that the conclusion remains valid if we use
the substitution $t_i \mapsto \lambda_i x^{n_i}$ for some arbitrary
numbers $\lambda_i \in \C$.

This viewpoint for instance suggests a generalization of the concept
of `fewnomial' to the case of powers of abelian
varieties. \footnote{None of the results of this paper is known in
  that case, and it seems of interest to ask whether an analogue of
  the R\'eny-Erd\H{o}s or Schinzel's conjecture are true in that
  context, already replacing $\Gm$ with an elliptic curve.} A relevant
result in this direction can be found in \cite{Lu2010}. But, more
important here, this is useful in the development of the proofs, and
it also suggests a number of links with other topics. We will discuss
in a moment those which appear to us more relevant.

\medskip

We also point out the following \emph{dichotomy between polynomials
  and rational functions}: we shall state a version of
\prettyref{thm:main-poly} for rational functions (see
\prettyref{thm:main}), where we drop the assumption that $f$ is monic,
and the conclusion will say that $g(x)$ can be written as the
\emph{ratio} of two polynomials with a bounded number of terms that
are \emph{not necessarily coprime}. An instance of this behavior is
the cyclotomic polynomial $g(x) = 1 + \dots + x^{n-1}$, which solves
the equation $(x^n-1) - g(x)(1-x) = 0$ without being a fewnomial, but
that in fact can be written as $g(x) = \frac{x^n - 1}{x-1}$ (observe
that the equation is indeed not monic).  This phenomenon is intrinsic
to the problem, and in fact many results here will be stated twice to
account for both polynomials and rational functions.

The suitable statement for rational functions can be deduced straight
away from \prettyref{thm:main-poly}; however, in this paper we shall
actually proceed in the opposite direction, first proving a theorem
for rational functions, and then recovering \prettyref{thm:main-poly}
(see the final argument in \prettyref{sec:induction}).

\subsection{Integral points on varieties over function fields.}
\label{sub:integral}

Many attractive Diophantine problems concern the $S$-integers
$\Oc_{S}$ and the $S$-units $\Oc_{S}^{*}$ in a number field $K$.
\footnote{We recall that
  $\Oc_{S}=\{x\in K:|x|_{v}\le1\ \forall v\not\in S\}$; for instance,
  for $K=\Q$, the $S$-units are those rationals with numerator and
  denominator made up only of primes in the finite set $S$.} The
latter may be also described as just the $S$-integral points for
$\Gm$. For instance, the \textit{Mordell-Lang conjecture for tori}
yields a description of the $S$-integral points on subvarieties $W$ of
$\Gm^{l}$, so the points on $W$ having $S$-unit coordinates. Such
description follows from the $S$-unit theorem of \textsc{J.H.\
  Evertse}, \textsc{H.P.\ Schlickewei} and \textsc{A.J.\ van der
  Poorten}, while the general conjecture for tori is a theorem of
\textsc{M.\ Laurent} since the '80s; see \cite[Thm.\
7.4.7]{Bombieri2006}.

Instead, much less is known for $S$-integral points on \textit{finite
  covers} of $\Gm^{l}$ (except for the case of curves). Take for
instance the simple-looking equation $y^{2}=1+x_{1}+x_{2}$, to be
solved with $x_{1},x_{2}\in\Oc_{S}^{*}$ and $y\in\Oc_{S}$. This
represents a double cover of $\Gm^{2}$, on which we seek the
$S$-integral points. Alternatively, they may be described as the
$S$-integral points for the affine variety obtained as the complement
in $\Pb_{2}$ of two lines and a suitable conic (see
\cite{Corvaja2008}).  Now, this is a divisor of degree $4$ with normal
crossings, so a celebrated conjecture of \textsc{P.\ Vojta} predicts
that the solutions are not Zariski-dense, but this has not yet been
proved (see \cite[\S 14.3]{Bombieri2006}; this special case was
proposed explicitly by \textsc{F.\ Beukers} in \cite{Beukers1995}).
\footnote{This is indeed a `borderline' case of Vojta's conjecture on
  integral points, one of the simplest but yet unsolved ones. See
  \cite{Corvaja2008} for a proof in the function field context. }

A related form of this problem has been recently proposed by
\textsc{D.~Ghioca} and \textsc{T.~Scanlon} while studying the
dynamical Mordell-Lang conjecture in positive
characteristic. Specifically, for a given prime $p$, they ask about
the integer solutions of $f(y)=c_{1}p^{a_{1}}+\cdots+c_{l}p^{a_{l}}$,
in the unknowns $y,a_{1},\ldots,a_{l}$, where the polynomial $f$ and
the constants $c_{1},\ldots,c_{l}$ are given. Since $p^{a_{i}}$ are
$S$-units, this is in turn a special case of seeking the integral
points on the cover of $\Gm^{l}$ given by $f(y)=x_{1}+\cdots+x_{l}$.

The methods so far known do not suffice even to treat the former
equation (see \cite{Corvaja2013a} for some special cases). Actually,
the problem arises even in writing down what is expected to be the
most general form of solution. Note that any identity of the shape
$f(g(x))=c_{1}x^{m_{1}}+\cdots+c_{l}x^{m_{l}}$, for a polynomial $g$,
would produce solutions simply by setting $x=p^{a}$.  Hence, it is a
primary task to write down all such identities. Note also that such an
identity (considered now over $\C$) represents an $S$-integral point
on the said cover, but now relative to the function field $\C(x)$ and
set $S=\{0,\infty\}$: \footnote{Here, in accordance with quite a
  general principle, the integral points over a function field may be
  used to parametrize integral points over a number field.} in fact,
the $S$-units of $\C(x)$ are precisely the monomials $cx^{m}$.

This example makes evident the connection of these topics on integral
points with the topic of fewnomials (and with the R\'enyi-Erd\H{o}s
and Schinzel's mentioned conjectures); indeed, in the case of the
problem of Ghioca and Scanlon a complete description \textit{in finite
  terms} of the relevant identities follows from Theorem 2 of
\cite{Zannier2008}.

The results of the present paper yield a corresponding description in
a rather more general situation. Namely, in dealing with an arbitrary
finite cover $\pi:W\to\Gm^{l}$, they allow us to parametrize all the
regular maps $\rho:\Gm\to W$ (i.e., the $S$-integral points on $W$,
with respect to the function field $\C(x)$ and set
$S=\{0,\infty\}$). \footnote{The case of more general function fields
  or even more general sets $S$ is not known to us and seems to
  present subtle difficulties; this happens already by taking
  $S=\{0,1,\infty\}$. See \cite{Corvaja2013} for some cases related to
  surfaces.}

\begin{thm}\label{thm:vojta}
  Let $\pi:W\to\Gm^{l}$ be a finite map. Then there exist a finite set
  $\Psi$ of regular maps $\psi:V\times\Gm^{s}\to W$, with $s=s_{\psi}$
  an integer and $V=V_{\psi}$ an affine algebraic variety, such that
  for every regular map $\rho:\Gm\to W$ there exist a $\psi\in\Psi$, a
  point $\xi\in V_{\psi}(\C)$ and a regular map $\gamma:\Gm\to\Gm^{s}$
  with $\rho=\psi_{\xi}\circ\gamma$.
\end{thm}

Here $\psi_{\xi}$ denotes the restriction of $\psi$ to
$\{\xi\}\times\Gm^{s}$. The special case $l=2$ of this theorem appears
as Theorem 5.1 in \cite{Corvaja2006}, in different phrasing and with a
completely different (and somewhat involved) proof.

We therefore see that any `$S$-integral point' factors through a map
$\psi_{\xi}:\Gm^{s}\to W$ of \textit{bounded degree}, in the sense
that the inverse image of a hyperplane section of $W$ has bounded
degree in $\Gm^{s}\subset\Pb_{s}$. This can be expressed in terms of
boundedness of the \textit{heights} of the integral points.  Such
conclusion, which is in a sense best-possible, proves Vojta's
conjectures for $W$ and the integral points in question. \footnote{See
  e.g.\ \cite[\S 14]{Bombieri2006} for a general formulation of
  Vojta's conjectures, especially over number fields.  For brevity we
  omit here any further detail or example.} As an application, we can
prove the following corollary.

\begin{cor}\label{cor:vojta}
  Suppose that the union of images of the regular non-constant maps
  $\rho:\Gm\to W$ is Zariski-dense. Then the branch locus of $\pi$ in
  $\Gm^l$ is invariant by translation by an algebraic subgroup of
  positive dimension.
\end{cor}

This is a useful condition which fits within a classification of
\textsc{Y.\ Kawamata} (see, for instance, the remark after Thm.\ 2 of
\cite{Corvaja2013} or \S 5.5.5 of the recent book by Noguchi-Winkelman
\cite{Noguchi2014}).

Moreover, this language makes it also more obvious how to prove the
special case $l = 1$ (we thank one of the anonymous referees for
pointing out this argument). Indeed, for $l = 1$ an integral point is
a regular map $\rho : \Gm \to W$ such that the composition
$\pi \circ \rho : \Gm \to W \to \Gm$ is the map
$x \mapsto \theta x^n$, so an isogeny composed with a
translation. Suppose that one such point exists. It follows easily
that the normalization of $W$ is in fact isomorphic to $\Gm$ itself,
and all the integral points can then be easily classified.

\subsection{A `Bertini Theorem' for covers of tori.}
\label{sub:bertini}

Consider again a (ramified) cover $\pi:W\to\Gm^{l}$, by which we mean
a dominant map of finite degree $e$ from the irreducible algebraic
variety $W/\C$. When $\Gm^{l}$ is replaced by the affine space
$\A^{l}$, a version of the Bertini Irreducibility Theorem asserts that
for $l>1$, if $H$ is a `general' hyperplane in $\A^{l}$, the fiber
$\pi^{-1}(H)$ is still irreducible. In the present context one may
replace $H$ by a `general' algebraic subgroup (or coset) of $\Gm^{l}$
and ask about the same conclusion. Of course, a marked contrast with
the Bertini case is that the algebraic subgroups now form a discrete
family, which prevents standard methods to work in this context. In
\cite[Thm.\ 3]{Zannier2010} a positive result was obtained, however
concerning irreducibility only above components of 1-parameter
subgroups, and not above arbitrary cosets.

Now, the arguments and results of this paper (completely independent
of \cite{Zannier2010}) directly lead to a toric analogue of Bertini's
Theorem without the said restriction.

\begin{thm}\label{thm:bertini}
  Let $W$ be a quasi-projective variety and $\pi:W\to\Gm^{l}$ be a
  (complex) dominant rational map of finite degree $e$, and suppose
  that the pullback $[e]^{*}W$ is irreducible. Let $X \subseteq W$ be
  a proper algebraic subvariety such that $\pi_{\vert W \setminus X}$
  is finite onto its image. Then there exists a finite union
  $\mathcal{E}=\mathcal{E}_{\pi, X}$ of proper algebraic subgroups of
  $\Gm^{l}$ such that if $H$ is a connected algebraic subgroup not
  contained in $\mathcal{E}$, then for all $\theta\in\Gm^{l}$,
  $\pi^{-1}(\theta H) \setminus X$ is irreducible.
\end{thm}

Note that if $\pi$ is already finite onto its image, then $X$ may be
omitted from the statement. However, as pointed out by an anonymous
referee, to whom we are grateful for the correction, in the general
case the subvariety $X$ must be included in the statement; for
instance, if $\pi : W \to \Gm^2$ is the blow-up of $\Gm^2$ at a point,
then the preimage of a coset passing through the point always contains
the exceptional divisor.  The hypothesis of irreducibility of the
pullback is also a necessary condition. \footnote{For instance, when
  $\pi$ is an isogeny of $\Gm^{l}$, the cover becomes reducible above
  every subgroup $\pi(H)$, for any torus $H$ not containing the kernel
  $K$ of $\pi$, since $\pi^{-1}(\pi(H))=HK$.}

As for \prettyref{thm:main-poly}, the set $\mathcal{E}$ can be given a
complete algorithmic description, which is rather uniform in the
data. Consider the following particular case. Suppose that the variety
$W$ can be represented as the hypersurface $f(t_1,\ldots,t_l,y)=0$,
with $\pi$ given by the projection on the first $l$
coordinates. Assume moreover that $f$ is a (Laurent) polynomial in the
$t_i$'s and monic in $y$. Under these assumptions, the map $\pi$ is
finite, and the conclusion of \prettyref{thm:main-poly} gives the
following strengthening.

\begin{namedthm}[Addendum to \prettyref{thm:bertini}]
  If $W$ is the hypersurface defined by $f(t_1, \dots, t_l, y) = 0$,
  where $f$ is a Laurent polynomial in $t_1, \dots, t_l$ and monic in
  $y$, and $\pi : W \to \Gm^{l}$ is the projection onto the first $l$
  coordinates, then $X = \emptyset$ and the set $\mathcal{E}$ may be
  chosen dependently only on $\deg(f)$.
\end{namedthm}

As an application, we immediately obtain the following corollary, in
which for a given integer $d>1$ we let $K_{d}(x)$ denote the
\emph{Kronecker substitution} $K_{d}(x)=(x,x^{d},\ldots,x^{d^{l-1}})$.
\begin{cor}
  \label{cor:kronecker}
  Let $f(t_{1},\ldots,t_{l},y)$ be a complex polynomial of degree
  $e>0$ in $y$ and such that $f(t_{1}^{e},\ldots,t_{l}^{e},y)$ is
  irreducible over $\C(t_{1},\ldots,t_{l})$.  Then $f(K_{d}(x),y)$ is
  irreducible over $\C(x)$ for all integers $d$ large enough in terms
  of $\deg(f)$.
\end{cor}

This had been obtained in \cite{Zannier2010} (with a completely
different proof), however without this uniformity, which was left as
an open question.

\subsection{An application to composite rational functions}
\label{sub:ratio}

One may propose an analogue for rational functions of the already
mentioned conjectures of Erd\H{o}s and subsequent ones by
Schinzel. Namely, let $f(x)$ be a rational function and suppose that
for a rational function $g(x)$, the composition $f(g(x))$ may be
written as a ratio of two polynomials (not necessarily coprime) with
at most $l$ terms.  \emph{Is there a $B=B(f,l)$ such that $g(x)$ may
  be represented as ratio of polynomials with at most $B$ terms?} The
present methods allow a positive solution of this problem as well, as
follows.

\begin{thm}\label{thm:ratfunc1}
  If $f,g\in \C(x) \setminus \C$ are such that the composition
  $f(g(x))$ can be written as the ratio $P(x)/Q(x)$, where
  $P,Q\in \C[x]$ have altogether at most $l$ terms, then there exist
  polynomials $p,q\in \C[x]$ with at most $B_{2}=B_{2}(l)$ terms such
  that $g(x)=p(x)/q(x)$.
\end{thm}

Again, we stress that the pairs $P, Q$ and $p, q$ are not necessarily
coprime. We also remark that we actually have full uniformity here in
the rational function $f$, as the number $B_{2}$ only depends on $l$
and not on $\deg(f)$ (this dependency can be removed thanks to a
previous theorem proved by the first and last authors
\cite{Fuchs2012}).

\subsection{Non-standard polynomials}

The notion of fewnomial and our main theorems can be translated
naturally in the language of \textsc{A.\ Robinson}'s non-standard
analysis. We refer the reader to \cite{Fried2008} for an introduction
to the subject.

Here we just recall that in non-standard analysis one has a map ${}^*$
which sends the standard objects, such as $\N$ or $\R$, to their
non-standard counterparts, in a way that preserves all first-order
formulas. The easiest example of (non-trivial) map ${}^*$ is the one
that sends any set $S$ into the set of sequences with values in $S$
(i.e., $S^\N$) modulo the equivalence relation defined by a fixed
non-principal ultrafilter on $\N$ (i.e., $(a_n)\sim(b_n)$ if
$\{n\ :\ a_n=b_n\}$ is in the ultrafilter). This introduces new,
non-standard elements; for instance, the non-standard ${}^*\N$
contains an element $\omega$, the equivalence class of the sequence
$(n)_{n\in\N}$, which is different from any standard natural number.

Concerning our context, we note that the non-standard ${}^*(\C[x])$
contains `polynomials with infinitely many terms', such as
\[
  1 + x + x^2 + \dots + x^{\omega-1} + x^\omega.
\]
In fact, this is exactly the equivalence class of the sequence
$(1+x+\dots+x^n)_{n\in\N}$.

We now define the \emph{ring $\mathcal{F}$ of fewnomials} in
${}^*(\C[x])$ to be the subring of polynomials whose number of terms
is actually finite:
\[
  \mathcal{F} := \{a_1x^{n_1} + \dots + a_lx^{n_l}\ :\ l\in\N,\
  a_i\in{}^*\C,\ n_i\in{}^*\N\}.
\]
In this language, statements about fewnomials become quite compact. As
an instance of this phrasing, the Erd\H{o}s-R\'enyi conjecture proved
by Schinzel becomes: \emph{if $g^2\in\mathcal{F}$ for some
  $g\in{}^*(\C[x])$, then $g\in\mathcal{F}$}. Likewise,
\prettyref{thm:main-simple} translates to the following quite short
statement:

\begin{namedthm}[Theorem~${}^*$\ref{thm:main-simple}]
  The ring $\mathcal{F}$ is integrally closed in ${}^*(\C(x))$.
\end{namedthm}

This statement was proposed by \textsc{A.\ Fornasiero} before the
results of this paper, together with its following immediate
corollary (which is, in turn, a non-standard translation of \prettyref{thm:main}, stated in the following section).

\begin{cor*}[Theorem~${}^{*}$\ref{thm:main}]
  The fraction field of $\mathcal{F}$ is relatively algebraically
  closed in ${}^{*}(\C(x))$.
\end{cor*}

The latter conclusion is another example of the aforementioned
behavior of rational functions, and indeed it corresponds to dropping
the assumption that the polynomial is monic in
\prettyref{thm:main-simple}.

It is rather easy to see that the Theorems ${}^*$\ref{thm:main-simple}
and \ref{thm:main-simple} are indeed equivalent. For example, assume
Theorem~${}^*$\ref{thm:main-simple} and suppose by contradiction that
\prettyref{thm:main-simple} is false. Then for some $d,l\in\N$ there
should be a sequence $(g_n(x))$ of polynomials whose number of terms
grows to infinity, while they also satisfy
\[ f_n(x,g_n(x)) = 0 \] where $(f_n)$ is a sequence of polynomials
with at most $l$ terms, of degree at most $d$ and monic in the last
variable.

But then the equivalence classes ${}^*g$ and ${}^*f$ of the above
sequences satisfy
\[
  {}^*f(x,{}^*g(x))=0,
\]
which means that ${}^*g(x)$ is integral over $\mathcal{F}$, while it
lies in
${}^*(\C[x])$ and not in $\mathcal{F}$, a contradiction.

\medskip

Although there are details to be worked out, we believe that also our
proof of \prettyref{thm:main-simple} can be translated rather
naturally to a shorter argument in the non-standard language; this is
being investigated and may appear in a future paper. On the one hand,
we would loose effectivity, but on the other, we may be able to avoid
the use of the resolution of singularities and directly use the
construction of Puiseux series.

The main potential simplification comes from the fact that many
notions, which in the proof depend on carefully chosen parameters,
become absolute. For example, the notion of being ``small'' with
respect to a ``large'' number, which in our proof depends on a
parameter $\varepsilon$ to be chosen carefully, translates to being
\emph{infinitesimal} with respect to the second number.

\subsection{Fewnomials and Unlikely Intersections}

This instance does not directly use results of the present paper, but
we still discuss it because it is far from being unrelated.

As already mentioned at the end of \prettyref{sub:main}, several
results here contain a dichotomy \emph{lacunary polynomials
  $\leftrightarrow$ lacunary rational functions}, where by the latter
terminology we mean rational functions which may be represented as a
ratio of two fewnomials, possibly non-coprime, as in
\prettyref{thm:ratfunc1}. Recall the standard example
$(x^{n}-1)/(x-1)$, which shows that a \textit{lacunary rational
  function} which is a polynomial is not necessarily a fewnomial.
This gives rise to the following problem, also posed independently by
\textsc{M.\ Zieve}.

Suppose that a rational function can be represented as
$r(x)=g(x^{n_{1}},\ldots,x^{n_{l}})/h(x^{n_{1}},\ldots,x^{n_{l}})$,
where the integers $n_{i}$ vary, while $g,h$ are \textit{fixed}
coprime polynomials in $\C[t_{1},\ldots,t_{l}]$. (In accordance with
the viewpoint illustrated above, we are viewing $r(x)$ as the
restriction of a fixed rational function $g/h$ on $\Gm^{l}$ to a
$1$-dimensional algebraic subgroup which may vary.) One may ask:

\begin{question}
  For which $1$-dimensional algebraic subgroups does $r(x)$ become a
  (Laurent) polynomial?
\end{question}

For instance, the above example comes from $g=t_{2}-1$, $h=t_{1}-1$ on
$\Gm^{2}$; in this case it is easy to check that the only
$1$-dimensional algebraic subgroups which make $g/h$ a (Laurent)
polynomial are given by $t_{2}=t_{1}^{n}$ for integer $n$ (as in the
example).

Therefore, in particular, we have two coprime polynomials $g,h$ such
that they become non-coprime (or such that $h$ becomes invertible)
along the $1$-dimensional subtorus of $\Gm^{l}$ parametrized by
$t_{i}\mapsto x^{n_{i}}$.  This kind of problem also appeared in a
conjecture of Schinzel, which was later recognized as a special case
of the more recent Zilber-Pink conjecture in the realm of the
so-called \textit{Unlikely Intersections}.  See \cite{Zannier2012} for
a discussion of this topic, especially Ch.\ 2. This conjecture of
Schinzel was confirmed by \textsc{E.\ Bombieri} and the third author
(see \cite[Appendix]{Schinzel2000}), and was later refined with other
methods, in collaboration also with \textsc{D.\ Masser}, in
\cite[Thm.\ 1.5]{Bombieri2007}, in a work proving the Zilber-Pink
conjecture for intersections with $1$-dimensional subgroups.

These last results give an answer to the above question, showing that
the relevant algebraic subgroups are contained in a \emph{finite union
  $\mathcal{E}=\mathcal{E}_{g,h}$ of proper algebraic subgroups of
  $\Gm^{l}$}. Given this, one may restrict to the subgroups in
$\mathcal{E}$ and continue by induction to write down all the
possibilities: it turns out that \emph{the relevant $1$-dimensional
  algebraic subgroups are precisely those contained in a certain
  finite union $\mathcal{E}'$ of proper algebraic subgroups on which
  $g/h$ becomes regular}.

\medskip

It is to be remarked that the more general question in which
\textit{$1$-dimensional algebraic subgroups} are replaced by
\textit{$1$-dimensional algebraic cosets} does not admit a similar
solution. This corresponds to the ratio
$g(\theta_{1}x^{n_{1}},\ldots,\theta_{l}x^{n_{l}})/h(\theta_{1}x^{n_{1}},\ldots,\theta_{l}x^{n_{l}})$
being a polynomial, for integers $n_{i}$ and nonzero constants
$\theta_{i}$.  We do not know of any method able to deal with such a
question in full generality.

\medskip

Another connection to integral points was pointed out to us, and we
are grateful for it, by one of the referees.  Starting from the
rational function
$r(t_1,\ldots,t_l) = g(t_1,\ldots,t_l) / h(t_1,\ldots,t_l)$, one can
blow-up the codimension-two subvariety of $\Gm^l$ defined by the
simultaneous vanishing of $g$ and $h$, and remove the strict transform
of the subvariety of $\Gm^{l}$ defined by $h = 0$. Let $W$ be the
resulting variety. Then the $1$-dimensional (translates of) algebraic
subgroups which are solutions to Question 1.8 correspond to regular
maps $\Gm \rightarrow W$. With this interpretation a solution for the
problem of translates can possibly be given for $l = 2$ (the case of
surfaces) under some normal-crossings conditions which will depend on
$h$ and are generically satisfied.

\subsection{Proof methods and quantitative issues.}

The strategy of the proofs here follows only in part the pattern of
\cite{Zannier2008}; this shall be outlined in more detail in
\prettyref{sec:strategy} (before the formal arguments). The main
technical issue is finding an appropriate way of expanding $g(x)$ as a
kind of multi-variate Puiseux series. This is done here by using first
the theory of resolution of singularities to reduce to a rather
regular case in which one can use multi-variate analytic
expansions. An earlier version of the proofs involved a different,
more complicated construction of certain Puiseux-type expansions, but
no use of resolution of singularities. The approach was dropped in
favor of the present one for the sake of simplicity, but it may be of
independent interest, and it can be still found in an earlier draft of
this paper \cite{Fuchs2014v1}.

A byproduct is a completely effective output of the proofs: \emph{one
  can obtain effective estimates for the involved quantities, and
  effective parametrizations} (provided of course one deals with cases
in which the fields and equations which occur are finitely
presented). However, we do not give here explicit bounds, which in any
case would have the shape of highly iterated
exponentials. \footnote{In the original cases of the R\'enyi-Erd\H{o}s
  conjecture, doubly exponential bounds had been obtained by Schinzel
  \cite{Schinzel1987}, reduced later to single exponential by Schinzel
  and the third author \cite{Schinzel2009}.}

\medskip

\subsection*{Acknowledgments}
We express our gratitude to A.\ Fornasiero for raising the question in
the non-standard setting, thus renewing interest in this problem, and
to D.\ Ghioca and T.\ Scanlon for informing us about their conjecture
and its link with the problems discussed here. We also wish to thank
the anonymous referees for the very detailed reading and the various
important comments, corrections and further pointers to the
literature.

\section{Variations and reductions}
\label{sec:variations}

\subsection{Variations of \prettyref{thm:main-poly}}

The following three statements are variations regarding irreducible
factors and the dichotomy rational functions $\leftrightarrow$
polynomials mentioned in the introduction.

The first one concerns factorizations.

\begin{thm}
  \label{thm:main-poly-factor}
  If $f\in\C[t_{1},\dots,t_{l},y]$ is monic in $y$ and of degree at
  most $d$ in each variable, if $n_{1},\dots,n_{l}$ are natural
  numbers, and if $g,h\in\C[x,y]$ are polynomials monic in $y$ such
  that
  \begin{equation}
    g(x,y)h(x,y)=f(x^{n_{1}},\dots,x^{n_{l}},y)
  \end{equation}
  then each coefficient of $g$ (as a polynomial in $y$) has at most
  $B_{3}=B_{3}(d,l)$ terms.
\end{thm}

(By symmetry, a similar conclusion holds automatically for the
coefficients of $h$.)

Note that we recover \prettyref{thm:main-poly} on taking $(y-g(x))$ as
the first factor.  The converse deduction is also not difficult but
shall be explained later. The other variations concern rational
functions.

\begin{warn}
  We stress again the point that whenever we write a rational function
  (even when it is a polynomial) as quotient of two polynomials, we
  are not usually assuming that the numerator and denominator are
  coprime (recall the example $(x^n - 1) / (x - 1)$). This issue is
  related to not requiring that $f$ is monic in $y$, as in the
  following statements.
\end{warn}

\begin{thm}
  \label{thm:main}
  If $f\in\C[t_{1},\dots,t_{l},y] \setminus \C$ is a polynomial of
  degree at most $d$ in each variable, if $n_{1},\dots,n_{l}$ are
  integers, and if $g(x)\in\C(x)$ is such that
  \begin{equation}
    f(x^{n_{1}},\dots,x^{n_{l}},g(x))=0,\label{eq:main}
  \end{equation}
  then $g(x)$ is the ratio of two polynomials in $\C[x]$ with at most
  $B_{4}=B_{4}(d,l)$ terms.
\end{thm}

\begin{thm}
  \label{thm:main-factor}
  If $f\in\C[t_{1},\dots,t_{l},y]$ is a polynomial of degree at most
  $d$ in each variable, if $n_{1},\dots,n_{l}$ are integers, and if
  $g,h\in\C(x)[y]$ are such that
  \begin{equation}
    g(x,y)h(x,y)=f(x^{n_{1}},\dots,x^{n_{l}},y)
  \end{equation}
  with $g$ monic in $y$, then each of the coefficients of $g$ (as a
  polynomial in $y$) is the ratio of polynomials in $\C[x]$ with at
  most $B_{5}=B_{5}(d,l)$ terms.
\end{thm}

It is easy to see that \prettyref{thm:main-poly} implies
\prettyref{thm:main}, but the converse deduction does not appear as
straightforward (we stress yet again the point that the polynomial $g$
in the conclusion of \prettyref{thm:main} is represented as a quotient
of two polynomials which need not to be coprime). In this paper, we
actually prove \prettyref{thm:main} first, and then deduce
\prettyref{thm:main-poly} (and \prettyref{thm:main-poly-factor}) via a
general integrality argument.

\begin{rem}
  In all of the above statements, we may actually allow
  $n_{1},\dots,n_{l}$ to be negative and $g(x)\in\C[x,x^{-1}]$, with a
  similar conclusion.

  We may also deduce that the fewnomials which arise can be
  parametrized with the same exponents. For instance, in
  \prettyref{thm:main-poly}, we can say that there are $N$ and
  $G\in\C[t_1,\dots,t_l]$, with $N$ and $\deg(G)$ bounded in terms of
  $d$ and $l$ only, such that $g(x^N)=G(x^{n_1},\dots,x^{n_l})$.

  For the sake of simplicity, we shall omit details about these
  further assertions.
\end{rem}

\subsection{Reductions}
\label{sub:reductions}

We are going to prove \prettyref{thm:main} first, and we can use some
standard arguments to reduce the theorem to a simpler situation. In a
moment, we shall reduce both Theorems \ref{thm:main},
\ref{thm:main-factor} about rational functions to the case where $f$
is monic in $y$, and $n_{1},\dots,n_{l}$ are non-negative. We obtain
the following statements, in which the assumption is as in
\prettyref{thm:main-poly} (or \prettyref{thm:main-poly-factor}), but
the conclusion is as in \prettyref{thm:main} (resp.\
\prettyref{thm:main-factor}).

\begin{prop}
  \label{prop:main}
  If $f\in\C[t_{1},\dots,t_{l},y] \setminus \C$ is monic in $y$ and of
  degree at most $d$ in each variable, if $n_{1},\dots,n_{l}$ are
  natural numbers, and if $g(x)\in\C(x)$ is such that
  \begin{equation}
    f(x^{n_{1}},\dots,x^{n_{l}},g(x))=0,\label{eq:prop-main}
  \end{equation}
  then $g(x)$ is the ratio of two polynomials in $\C[x]$ with at most
  $B_{6}=B_{6}(d,l)$ terms.
\end{prop}

Note that since $f$ is monic in $y$, it follows that $g(x)$ is
actually a polynomial, but the conclusion only says that $g$ is
represented by a quotient of two polynomials which need not be
coprime. Thus this proposition is a weak form of
\prettyref{thm:main-poly}. Similarly for its corollary.

\begin{prop}
  \label{prop:main-factor}
  If $f\in\C[t_{1},\dots,t_{l},y]$ is monic in $y$ and of degree at
  most $d$ in each variable, if $n_{1},\dots,n_{l}$ are natural
  numbers, and if $g,h\in\C(x)[y]$ are such that
  \begin{equation}
    g(x,y)h(x,y)=f(x^{n_{1}},\dots,x^{n_{l}},y)\label{eq:prop-main-factor}
  \end{equation}
  with $g$ monic in $y$, then the coefficients of $g$ as a polynomial
  in $y$ are the ratios of polynomials in $\C[x]$ with at most
  $B_{7} = B_{7}(d,l)$ terms.
\end{prop}

Both are clearly special cases of the original Theorems \ref{thm:main}
and \ref{thm:main-factor}. As we now show, it is not difficult to
deduce the latter statements from them.

\begin{note}
  \label{note:deductions-l}
  It is important to note that \emph{the following deductions are
    valid for each single $l$} (whereas the number $d$ is changed in
  the course of the deductions). This is crucial in all our proofs
  that proceed by induction on $l$; namely, if we assume that one
  statement is true for a certain value of $l$ and all possible $d$'s,
  the other statements will follow as well for the same value of $l$
  and all possible $d$'s.
\end{note}

\begin{proof}[Deduction\ of\ \prettyref{thm:main}\ from\
  \prettyref{prop:main}]
  Note that \prettyref{prop:main} requires $n_1, \dots, n_l$ to be
  natural numbers rather than integers. We may reduce to the case
  $n_i \geq 0$ by replacing, when necessary, $t_i$ by $t_i^{-1}$ and
  multiplying the resulting polynomial by $t_i^{d}$; after this
  transformation, the degree of $f$ in each variable is still bounded
  by $d$. Therefore, we may assume that $n_i \geq 0$ for all $i$.

  Write $f$ as
  \[
    f=\sum_{i=0}^{d}h_{i}(t_{1},\dots,t_{l})y^{i}
  \]
  where the $h_{i}$'s are polynomials of degree at most $d$ in each
  variable. Let $e\leq d$ be the maximum integer such that
  $h_{e}(x^{n_{1}},\dots,x^{n_{l}})$ is not identically zero, and let
  $f_{1}:=\sum_{i=0}^{e}h_{i}(t_{1},\dots,t_{l})y^{i}$.

  We now consider the polynomial
  $f_{2}:=h_{e}^{e-1}f_{1}(t_{1},\dots,t_{l},y/h_{e})$. Note that
  $f_{2}$ is monic in $y$, and it has degree at most
  $(e-1)d+d\leq d^{2}$ in each variable. Assuming
  \prettyref{prop:main}, each rational root of
  $f_{2}(x^{n_{1}},\dots,x^{n_{l}},y)$ is the ratio of two polynomials
  with at most $B_{6}(d^{2},l)$ terms.  Multiplying each such root by
  $h_{e}$ we obtain all the rational roots of
  $f(x^{n_{1}},\dots,x^{n_{l}},y)$, and therefore the rational
  solutions of \prettyref{eq:main}. In particular, the solutions are
  ratios of polynomials with at most
  $B_{4}(d,l):=(d+1)^{l}B_{6}(d^{2},l)$ terms, as desired.
\end{proof}

\begin{proof}[Deduction\ of\ \prettyref{thm:main-factor}\ from\
  \prettyref{prop:main-factor}]
  We proceed as in the previous proof to show that
  $B_{5}(d,l):=(d+1)^{dl}B_{7}(d^{2},l)$ is a suitable value for
  $B_{5}$.
\end{proof}

Moreover, as promised earlier, we can easily deduce
\prettyref{thm:main-factor} from \prettyref{thm:main}. Thanks to the
above reductions, it is sufficient to deduce
\prettyref{prop:main-factor} from \prettyref{prop:main}.

\begin{proof}[Deduction\ of\ \prettyref{prop:main-factor}\ from\
  \prettyref{prop:main}]
  Suppose that $p(x)$ is a coefficient of a monic irreducible factor
  of the polynomial monic in $y$
  \[
    \phi(x,y):=f(x^{n_{1}},\dots,x^{n_{l}},y).
  \]

  Let us call $\alpha_{1},\dots,\alpha_{e}$ the roots of this
  polynomial in an algebraic closure of $\C(x)$, with repetitions,
  where $e=\deg_{y}\phi=\deg_{y}f$. The polynomial $p(x)$ is, up to
  sign, an elementary symmetric polynomial in some of the roots. Let
  us denote the elementary symmetric polynomials as
  $\Sigma_{j}^{k}(z_{1},\dots,z_{k}):=\sum_{1\leq
    i_{1}<\dots<i_{j}\leq k}z_{i_{1}}\cdot\ldots\cdot z_{i_{j}}$.

  Up to reordering the roots, we may write
  \[
    p(x)=\pm\Sigma_{j}^{k}(\alpha_{1},\dots,\alpha_{k})
  \]
  for some $0\leq j\leq k\leq e$. This implies that $p(x)$, up to
  sign, is a root of the monic polynomial
  \[
    \psi_{jk}(x,y):=\prod_{1\leq i_{1}<\dots<i_{k}\leq
      e}\left(y-\Sigma_{j}^{k}(\alpha_{i_{1}},\dots,\alpha_{i_{k}})\right).
  \]
  But the coefficients of $\psi_{jk}$ are now symmetric polynomials in
  the roots $\alpha_{i}$, which implies that they are actually
  polynomials in the $\Sigma_{i}^{e}$'s, i.e., the coefficients of
  $\phi$. A rough estimate shows that the degree of each such
  polynomial in each variable is at most $e^{2}\leq d^{2}$.

  This implies that we may find
  $f_{j,k}(t_{1},\dots,t_{l},y)\in\C[t_{1},\dots,t_{l},y]$ monic in
  $y$ and of degree at most $d^{2}$ in each variable such that
  \[
    f_{j,k}(x^{n_{1}},\dots,x^{n_{l}},y)=\psi_{j,k}(x,y).
  \]

  Assuming \prettyref{prop:main}, since $p(x)$ is a root of
  $\psi_{j,k}$, it must be a ratio of two polynomials with at most
  $B_{7}(d,l):=B_{6}(d^{2},l)$ terms, as desired.
\end{proof}

The exact same argument can be also used to show that
\prettyref{thm:main-poly-factor} follows from
\prettyref{thm:main-poly}.

\begin{proof}[Deduction\ of\ \prettyref{thm:main-poly-factor}\ from\
  \prettyref{thm:main-poly}]
  We proceed as in the previous proof to show that $B_{1}(d^{2},l)$ is
  a suitable value for $B_{3}$.
\end{proof}

\subsection{Further lemmas}

In the course of our proof, we will need on few occasions to replace
$x$ with an auxiliary variable $x_{n}$ such that $x_{n}^{n}=x$. In the
next lemma, we show that these substitutions do not affect our
statements, so they may be considered as immaterial.

\begin{lem}
  \label{lem:ratio-n-to-1}Let $g(x)$ be a polynomial such that
  $g(x^{n})$ can be written as the ratio of two polynomials with at
  most $B$ terms. Then $g(x)$ is the ratio of two polynomials in
  $\C[x]$ with at most $B$ terms.
\end{lem}
\begin{proof}
  Suppose that $g(x^{n})=\frac{p(x)}{q(x)}$, where $p$ and $q$ are
  polynomials with at most $B$ terms. Grouping the monomials whose
  degrees in $x$ are in the same congruence class modulo $n$ we may
  (uniquely) write
  \[
    p(x)=p_{0}(x^{n})+xp_{1}(x^{n})+\dots,\quad
    q(x)=q_{0}(x^{n})+xq_{1}(x^{n})+\dots
  \]
  with $p_{i}$, $q_{i}$ polynomials with at most $B$ terms as well.

  But then, since $g(x^{n})q(x)=p(x)$, we must have
  $g(x^{n})q_{i}(x^{n})=p_{i}(x^{n})$ for all $i$, and in particular
  $g(x)q_{i}(x)=p_{i}(x)$. As at least one $q_{i}$ is non-zero, we
  have found a representation of $g(x)$ as the ratio of two
  polynomials with at most $B$ terms, as desired.
\end{proof}

Another easy reduction shows that if we find a $\Z$-linear relation
with bounded coefficients between the exponents $n_{1},\dots,n_{l}$,
then we may actually remove one of the exponents. This is also crucial
for our induction on $l$.

\begin{lem}
  \label{lem:linear-dep-exp} Suppose that we are under the hypothesis
  of \prettyref{thm:main}, and that there are integers
  $h_{1},\dots,h_{l}$, not all zero, and some $C>0$ such that
  \[
    h_{1}n_{1}+\dots+h_{l}n_{l}=0,\:|h_{i}|\leq C.
  \]
  Assume moreover that \prettyref{thm:main} has been proved for
  $(l-1)$ and any degree $d$. Then $g(x)$ is the ratio of two
  polynomials with at most $B_{4}(2dC,l-1)$ terms.
\end{lem}
\begin{proof}
  Without loss of generality, we may assume that $h_{l}\neq0$. In this
  case, we take new variables $u_{1},\dots,u_{l-1}$, we replace
  $t_{i}$ in $f$ with $u_{i}^{h_{l}}$ for $i=1,\dots,l-1$ and $t_{l}$
  with $u_{1}^{-h_{1}}\cdots u_{l-1}^{-h_{l-1}}$, and we multiply the
  result by $\left(u_{1}\cdots u_{l}\right)^{dC}$. The resulting
  polynomial has degree at most $2dC$ in each variable, and it
  vanishes at $u_{i}=x^{n_{i}}$ and $y=g(x^{h_{l}})$.

  Now, using the assumption about \prettyref{thm:main} and
  \prettyref{lem:ratio-n-to-1}, $g(x)$ is the ratio of two polynomials
  with at most $B_{5}(2dC,l-1)$ terms.
\end{proof}

\section{Introduction to the proof}
\label{sec:strategy}

In order to prove \prettyref{thm:main}, we build up on the same
technique of \cite{Zannier2008} but with the additional use of the
theory of resolution of singularities to reduce to a sufficiently
regular case. Indeed, the underlying expansions depend not quite on
the variable $x$, but on the $l$ variables $t_{1},\dots,t_{l}$; it is
well known that expansions of algebraic functions of several variables
often depend on subtle geometric features.

For the sake of illustration, we explain the strategy of the proof in
a simpler example where this combinatorial aspect is missing. We work
by induction on $l$.

Say that, as in the original Erd\H{o}s' conjecture (a special case of
\prettyref{thm:main-poly}), we start with the polynomial
\[
  f(t_{1},\dots,t_{l},y)=y^{2}-c_{0}-c_{1}t_{1}-\dots-c_{l}t_{l}.
\]
For simplicity, we also assume that $c_{0}=1$.

If we want to prove that a rational root $g(x)$ of
\[
  f(x^{n_{1}},\dots,x^{n_{l}},y)=\phi(x,y)=y^{2}-1-c_{1}x^{n_{1}}-\dots-c_{l}x^{n_{l}}
\]
is the ratio of two polynomials with few terms, we may expand $g(x)$
with the binomial series; namely, letting
$h(x):=c_{1}x^{n_{1}}+\dots+c_{l}x^{n_{l}}$, we may easily obtain the
multinomial expansion
\[
  g(x)=1+\frac{h(x)}{2}-\frac{h(x)^{2}}{8}+\dots=\sum_{k_{1}=0}^{\infty}\dots\sum_{k_{l}=0}^{\infty}c_{k_{1},\dots,k_{l}}x^{k_{1}n_{1}+\dots+k_{l}n_{l}}.
\]

It is crucial that $k_{1},\dots,k_{l}$ run through natural
numbers. Assuming that $0<n_{1}\leq n_{2}\leq\dots\leq n_{l}$, if
$n_{1}\ge\varepsilon n_{l}$ for some fixed $\varepsilon>0$, each
exponent $k_{1}n_{1}+\dots+k_{l}n_{l}$ is at least
$(k_{1}+\dots+k_{l})\varepsilon n_{l}$.  Since the degree of $g(x)$
must be $(n_{l}/2)$, we find that all terms must eventually cancel
except possibly for those such that
$(k_{1}+\dots+k_{l})\leq1/(2\varepsilon)$, leading to the bound
$(2\varepsilon)^{-l+1}/l!$ for the number of terms.

This consideration always works for $l=1$ (with $\varepsilon=1$), and
in particular we obtain the base case of our induction. However, in
general we have no lower bound at all for $n_{1}/n_{l}$. To cope with
this difficulty, the principle in \cite{Zannier2008} is that if some
terms $n_{1}, \dots, n_{p}$ are very small compared to $n_{l}$, we can
group together these small contributions as follows: we define
\[
  \delta(x)=1+c_{1}x^{n_{1}}+\dots+c_{p}x^{n_{p}},\;
  h_{1}(x)=c_{p+1}x^{n_{p+1}}+\dots+c_{l}x^{n_{l}}
\]
and we expand $g(x)$ as
\begin{equation}
  g(x)=\sqrt{\delta(x)}\left(1+\frac{h_{1}(x)}{\delta(x)}\right)^{1/2}=\sqrt{\delta(x)}\left(1+\frac{h_1(x)}{2\delta(x)}-\frac{h_1(x)^{2}}{8\delta(x)^{2}}+\cdots\right).\label{eq:pp-sqrt}
\end{equation}
As before, we can expand the powers of $h_{1}(x)$, which involve the
large exponents only; however, the new coefficients will not be
constants, as before, but actually functions in the hyperelliptic
function field $\C(x,\delta(x)^{1/2})$. Despite this radically new
feature, a theorem in Diophantine approximation over function fields
(see \prettyref{sec:approx}) allows one to reduce to the inductive
hypothesis at $p<l$, provided $n_{p+1}$ is large enough, by which we
mean that it is greater than $\varepsilon n_{l}$ for an absolute
$\varepsilon>0$. Of course, for some $0\leq p<l$ we must indeed have
that $n_{p}$ is small whereas
$n_{p+1}$ is large, concluding the argument.

\medskip

In the general case, we wish to apply the same approximation
technique. However, a direct attempt at expanding $g(x)$ as a kind of
multivariate series fails in the general case. The main issue is that
we may have monomials involving exponents that are combinations of
$n_{1},\dots,n_{l}$ with \emph{negative} coefficients, in which case a
combination of large exponents may become small, and it is not as easy
any more to separate the big ones from the small ones. These obstacles
appear when $g(0)$ is a non-simple root of $f({\bf
  0},y)$. \footnote{These issues are entirely avoided in the cases
  considered in \cite{Zannier2008}, where multinomial expansions
  suffice.}

We shall overcome these obstacles by applying a suitable monoidal
transformation to our original equation. Although $g(0)$ might still
be a non-simple root of $f({\bf 0}, y)$, the transformation will
guarantee that $g(x)$ can still be expanded as in the original
case. The choice of the monoidal transformation relies on the theory
of resolution of singularities. \footnote{We recall that an earlier
  draft of this papers contained a different proof based on a careful
  construction of a Puiseux-type expansion rather than resolution of
  singularities \cite{Fuchs2014v1}.}

\section{Reduction to the regular case}
\label{sec:regular}

In order to obtain our desired expansion of $g(x)$ as a
``pseudo-analytic series'', we prove that \prettyref{prop:main} can be
further reduced to a special case in which the polynomial $f$ is
sufficiently regular.

Let $f \in \C[t_1, \dots, t_l, y]$ be as in
\prettyref{prop:main}. Assume, as we may, that $f$ is irreducible. Let
$\C(t_1, \dots, t_l, z)$ be the function field generated by the
independent variables $t_1, \dots, t_l$ and an algebraic function $z$
such that $f(t_1, \dots, t_l, z) = 0$.

Let $W$ be a projective non-singular model of the function field
$\C(t_1, \dots, t_l, z)$. For each $i = 1, \dots, l$, let
$\mathcal{D}_i$ be the set of the irreducible components of the
divisor of $t_i$, and let
$\mathcal{D} := \bigcup_{i=1}^l \mathcal{D}_i$. By the known theory of
resolution of singularities, after applying some blow-ups, we may
further assume that \emph{the divisors appearing in $\mathcal{D}$ are
  non-singular and have normal crossings}.

Let $\phi : \Pb_1 \to W$ be the unique non-constant map such that:
\begin{itemize}
\item $t_i \circ \phi = x^{n_i}$ for all $i=1, \dots, l$;
\item $z \circ \phi = g(x)$,
\end{itemize}
where $x$ is the standard coordinate function $x : \Pb_1 \to
\Pb_1$. Let $P := \phi(0)$.

\begin{defn}\label{def:regular}
  Under the above notations, we say that a solution
  $n_1, \dots, n_l, g(x)$ to (\ref{eq:main}), namely
  $f(x^{n_1}, \dots, x^{n_l}, g(x)) = 0$, is \emph{regular} if
  $t_1, \dots, t_l$ are local parameters at $P$.
\end{defn}

Note that one might reformulate the above notion in a more compact and
geometric way by only referring to the map $\phi$. However, we prefer
to keep an explicit reference to the polynomial $g(x)$ and the numbers
$n_1, \dots, n_l$.

\medskip

The purpose of this section is to show that it suffices to prove the
conclusion of \prettyref{prop:main} in the special case in which the
solution is regular.

In what follows, given a divisor $D$ and a regular function $u$, we
let $v_D(u)$ denote the order of $u$ at $D$. Note that since $f$ has
degree at most $d$ in each variable, we have $|v_D(t_i)| \leq d^{l}$
for all $D \in \mathcal{D}$.

\begin{lem}\label{lem:invertible}
  Let $D_1, \dots, D_m$ be the divisors in $\mathcal{D}$ on which $P$
  lies. Then either
  $$
  h_1n_1 + \dots + h_ln_l = 0
  $$
  for some integers $h_1, \dots, h_l \in \Z$ not all zero such that
  $|h_i| \leq (d^{l}l)^{l}$, or $m = l$ and the matrix
  $(v_{D_i}(t_j))_{i,j}$ is invertible.
\end{lem}
\begin{proof}
  Since the divisors in $\mathcal{D}$ have normal crossings, we know
  at once that $m \leq l$. Assume that the matrix
  $(v_{D_i}(t_j))_{i,j}$ is not invertible, otherwise we are
  done. Then there are integers $h_1, \dots, h_l \in \Z$, not all
  zero, such that
  $$
  h_1v_{D_i}(t_1) + \dots + h_lv_{D_i}(t_l) = 0
  $$
  for all $i = 1, \dots, m$. Since $|v_{D_i}(t_j)| \leq d^{l}$ for all
  $i,j$, we may choose the integers $h_i$ so that
  $|h_i| \leq (d^{l}l)^{l}$ by Siegel's lemma.

  Let $u := t_1^{h_1} \cdot \dots \cdot t_l^{h_l}$. By the above
  observation, none of $D_1, \dots, D_m$ is a component of the divisor
  of $u$. On the other hand, the components of the divisor of $u$ are
  in $\mathcal{D}$. It follows that no such component contains $P$, so
  $u$ is regular at $P$ and $u(P) \in \C^{*}$. Note moreover that
  $u \circ \phi = x^{h_1n_1 + \dots + h_ln_l}$. Therefore,
  $$
  u(P) = x(P)^{h_1n_1 + \dots + h_ln_l} \in \C^{*}.
  $$
  Since $x(P) = 0$, it immediately follows that
  $h_1n_1 + \dots + h_ln_l = 0$, reaching the desired conclusion.
\end{proof}

Thanks to the above observation, in order to prove
\prettyref{prop:main}, it shall be sufficient to prove the following
special version which has a few more hypotheses.

\begin{prop}\label{prop:main-regular}
  If $f\in\C[t_{1},\dots,t_{l},y] \setminus \C$ is monic in $y$,
  irreducible, and of degree at most $d$ in each variable, and if
  $n_{1},\dots,n_{l} \in \N^{*}$ and $g(x)\in\C(x)$ form a regular
  solution of
  $$
  f(x^{n_{1}},\dots,x^{n_{l}},g(x))=0,
  $$
  then $g(x)$ is the ratio of two polynomials in $\C[x]$ with at most
  $B_8=B_8(d,l)$ terms.
\end{prop}

\begin{note}
  As in \prettyref{note:deductions-l}, the following deduction is
  valid at every single $l$.
\end{note}

\begin{proof}[Deduction of \prettyref{prop:main} from
  \prettyref{prop:main-regular}]
  We work by induction on $l$. In particular, if $n_i = 0$ for some
  $i$, we may specialize the variable $t_i$ to $1$; if $l > 1$, we
  conclude by inductive hypothesis, while if $l = 1$, we simply note
  that we actually have $g(x) \in \C$. Therefore, we may assume that
  $n_i \neq 0$ for all $i = 1, \dots, l$. Moreover, we may replace $f$
  by an irreducible factor, and therefore assume directly that $f$ is
  irreducible.

  By \prettyref{lem:invertible}, either $P$ lies on $l$ distinct
  divisors $D_1, \dots, D_l$ such that the matrix
  $(v_{D_i}(t_j))_{i,j}$ is invertible, or there is a relation
  $h_1n_1 + \dots + h_ln_l = 0$ with integers $h_i$ not all zero and
  such that $|h_i| \leq (d^{l}l)^{l}$. In the latter case, we must
  have $l > 1$, and we may conclude by \prettyref{lem:linear-dep-exp}
  and the inductive hypothesis. Therefore, we may assume to be in the
  former case.

  Let $u_1, \dots, u_l$ be new independent variables, and set
  $$
  f_1(u_1, \dots, u_l, y) := f\left(u_1^{v_{D_1}(t_1)} \cdot \dots
    \cdot u_l^{v_{D_1}(t_l)}, \dots, u_1^{v_{D_l}(t_1)} \cdot \dots
    \cdot u_l^{v_{D_l}(t_l)}, y\right).
  $$
  Note that $f_1$ is a polynomial of degree at most $d^{l+1}l$ in each
  variable.

  Let $(r_{i,j})_{i,j}$ be the inverse matrix of
  $(v_{D_i}(t_j))_{i,j}$ multiplied by its determinant $\Delta$, so
  that its coefficients are all in $\Z$.  Let
  $m_i := r_{1,i}n_1 + \dots + r_{l,i}n_l$. By construction, we have
  $$
  f_1(x^{m_1}, \dots, x^{m_l}, g(x^{\Delta})) = f(x^{n_1\Delta},
  \dots, x^{n_l\Delta}, g(x^{\Delta})) = 0.
  $$
  In turn, we choose an irreducible factor $f_2$ of $f_1$ such that
  \begin{equation}
    f_2(x^{m_1}, \dots, x^{m_l}, g(x^{\Delta})) = 0.\label{eq:f2}
  \end{equation}

  We claim that $m_1, \dots, m_l, g(x^{\Delta})$ form a regular
  solution of this equation in the sense of
  \prettyref{def:regular}. Indeed, we may now assume that
  $u_1, \dots, u_l$ are algebraic functions in some algebraic closure
  of $\C(t_1, \dots, t_l, z)$ such that
  $$
  t_i = u_1^{v_{D_i}(t_1)} \cdot \dots \cdot u_l^{v_{D_i}(t_l)}
  $$
  and moreover $f_2(u_1, \dots, u_l, z) = 0$. Let $W'$ be a projective
  non-singular model of the function field $\C(u_1, \dots, u_l, z)$,
  equipped with a surjective, finite map $\pi : W' \to W$. As at the
  beginning of the section, we may apply some blow-ups and assume that
  all the components of the divisors of the functions $u_i$ are
  non-singular and have normal crossings.

  Let $\phi' : \Pb_1 \to W'$ be the unique non-constant map such that
  $u_i \circ \phi' = x^{m_i}$ and $z \circ \phi' = g(x^{\Delta})$. Let
  $P' := \phi'(0) \in W'$. Since by construction
  $\pi \circ \phi' = x^{n_i\Delta}$, it follows at once that
  $\pi(P') = P$. For $i = 1, \dots, l$, let $E_i$ be a component of
  $\pi^{*}(D_i)$ on which $P'$ lies.

  Finally, note that
  $u_i^{\Delta} = t_1^{r_{1,i}} \cdot \dots \cdot t_l^{r_{l,i}}$. In
  particular, $u_i^{\Delta}$ can be factored as $u_i' \circ \pi$,
  where $u_i'$ is a function on $W$. By construction, we have
  $v_{D_i}(u_j') = \Delta\delta_{ij}$, where $\delta_{ij}$ is the
  Kronecker delta. Let $w_1, \dots, w_l$ be local parameters of
  $D_1, \dots, D_l$, so that $u_i' \sim w_i^{\Delta}$ (in the sense of
  analytic equivalence). It follows at once that
  $u_i \sim w_i \circ \pi$; since the function field extension is
  generated by the functions $u_i$, each $w_i \circ \pi$ is also a
  local parameter of $E_i$, and in particular, $u_i$ is a local
  parameter of $E_i$. Since the divisors $E_i$ have normal crossings,
  this means that $u_1, \dots, u_l$ are local parameters at $P'$.

  In turn, $m_1, \dots, m_l, g(x^{\Delta})$ form a regular solution of
  (\ref{eq:f2}), so $g(x^{\Delta})$ can be written as the ratio of two
  polynomials with at most $B_8(d^{l+1}l,l)$ terms. By
  \prettyref{lem:ratio-n-to-1}, $g(x)$ can also be written as the
  ratio of two polynomials with at most $B_8(d^{l+1}l,l)$, concluding
  the argument.
\end{proof}

\section{From multivariate expansions to algebraic approximations}
\label{sec:multivariate}

We now start our argument towards the proof of
\prettyref{prop:main-regular}. From now up to
\prettyref{sec:induction}, assume that we are working under the
assumptions of \prettyref{prop:main-regular}, and in particular that
$n_1,\dots,n_l \in \N^{*}$ and $g(x) \in \C[x]$ form a regular
solution of $f(x^{n_1}, \dots, x^{n_l}, g(x)) = 0$, where
$f \in \C[t_1,\dots,t_l,y] \setminus \C$ is irreducible, monic in $y$,
and of degree at most $d$ in each variable.

Under the notation of \prettyref{sec:regular}, regularity means that
$t_1, \dots, t_l$ are local parameters at $P = \phi(0)$, which means
that there is an embedding of the regular functions at $P$ into
$\C[[t_1, \dots, t_l]]$. Since moreover the function $z$ is integral
over $\C[t_1, \dots, t_l]$, we obtain an embedding
$$
\C[t_1, \dots, t_l, z] \hookrightarrow \C[[t_1, \dots, t_l]].
$$
The fairly trivial, but crucial observation, is that for any
$p = 0, \dots, l - 1$ we can also rewrite
$$
\C[[t_1, \dots, t_l]] \cong \C[[t_1, \dots, t_p]][[t_{p+1}, \dots,
t_l]].
$$

Therefore, fix one such $p = 0, \dots, l - 1$. We write $\bt$ for the
vector $(t_{p+1}, \dots, t_l)$. If $\bk$ is a vector of integers
$\bk = (k_{p+1}, \dots, k_l)$, we write
$\bt^{\bk} := t_{p+1}^{k_{p+1}} \cdot \dots \cdot t_l^{k_l}$. With
this notation, the above embedding yields a (unique) expansion
$$
z = \sum_{\bk \in \N^{l - p}} \alpha_{\bk}\bt^{\bk},
$$
where $\alpha_{\bk} \in \C[[t_1, \dots, t_p]]$. Recall that
\begin{equation}
  f\bigg(t_1, \dots, t_l, \sum_{\bk \in \N^{l - p}}
  \alpha_{\bk}\bt^{\bk}\bigg) = 0. \label{eq:series-t}
\end{equation}

We now specialize the above expansion along the curve $\phi(\Pb_1)$
and pull it back to $\Pb_1$. Recall that the ring of functions on
$\Pb_1$ that are regular at the origin can be embedded (uniquely) into
$\C[[x]]$. Under this embedding, the specialization and the pullback
simply mean that we specialize at $t_i = x^{n_i}$ term by term.

We first specialize at $t_i = x^{n_i}$ for $i = 1, \dots, p$. Since
$n_i \neq 0$ for all $i$, each series $\alpha_{\bk}$ converges at
$t_i = x^{n_i}$ to a series $\tilde{\alpha}_{\bk} \in \C[[x]]$, and we
have
\begin{equation}
  f\bigg(x^{n_1}, \dots, x^{n_p}, t_{p+1}, \dots, t_l, \sum_{\bk \in
    \N^{l - p}} \tilde{\alpha}_{\bk}\bt^{\bk}\bigg) = 0.\label{eq:series-x-t}
\end{equation}

Likewise, we can further specialize at $t_{i} = x^{n_{i}}$ for
$i = p+1, \dots, l$. We then obtain the expansion
\begin{equation}
  g(x) = \sum_{\bk \in \N^{l - p}} \tilde{\alpha}_{\bk}x^{\bk \cdot
    \bn}\label{eq:pseudo-analytic},
\end{equation}
where $\bn = (n_{p+1}, \dots, n_{l})$ and $\bk \cdot \bn$ denotes the
usual scalar product. Note that such expansion is convergent again
since $n_{i} \neq 0$ for all $i$.

Equation (\ref{eq:pseudo-analytic}) yields an expansion of $g(x)$
resembling an analytic expansion, but with coefficients that are
themselves functions of $x$, providing the first ingredient towards
the proof of \prettyref{prop:main-regular}. We now use
(\ref{eq:series-t}) and (\ref{eq:series-x-t}) to deduce some bounds on
the coefficients $\alpha_{\bk}$ and $\tilde{\alpha}_{\bk}$.

\begin{prop}\label{prop:degree-coeffs}
  The coefficients $\alpha_{\bk}$ generate a finite extension of
  degree at most $d$ of $\C(t_1, \dots, t_p)$. Similarly, the
  coefficients $\tilde{\alpha}_{\bk}$ generate a finite extension of
  degree at most $d$ of $\C(x)$.
\end{prop}
\begin{proof}
  Suppose that the coefficients generate either an algebraic extension
  of degree greater than $d$, or a non-algebraic extension. In both
  cases, we can apply Galois automorphisms over $\C(t_1, \dots, t_p)$
  to find at least $d+1$ distinct sequences of coefficients. In turn,
  such automorphisms extend naturally to $\C[[t_1, \dots, t_l]]$ by
  leaving $t_{p+1}, \dots, t_l$ fixed. Applying the automorphisms to
  (\ref{eq:series-t}), we find that the degree of $f$ in the last
  variable should be at least $d+1$, a contradiction. The conclusion
  for the coefficients $\tilde{\alpha}_{\bk}$ can be proved with a
  similar argument applied to (\ref{eq:series-x-t}) (just recall that
  $f$ is monic in $y$, so it does not become trivial when specializing
  $t_i = x^{n_i}$).
\end{proof}

For $L \in \N$, let $F_L$ be the field generated by
$\{\alpha_{\bk} \,:\, |\bk| \leq L\}$ over $\C(t_1, \dots, t_p)$,
where $|\bk|$ is the $1$-norm of $\bk \in \N^{l-p}$, and
$F_{\infty} := \bigcup_{L \in \N}F_L$. By
\prettyref{prop:degree-coeffs}, $[F_{\infty}:\C(x)] \leq
d$. Similarly, let $\tilde{F}_L$ be the field generated by
$\{\tilde{\alpha}_{\bk} \,:\, |\bk| \leq L\}$ over $\C(x)$, and
$\tilde{F}_{\infty} := \bigcup_{L \in \N}F_L$. Again,
$[\tilde{F}_{\infty}:\C(x)] \leq d$. Let $h$ be the logarithmic height
of the function field $\tilde{F}_{\infty}/\C$ normalized so that
$h(x) = [\tilde{F}_{\infty}:\C(x)] \leq d$.

\begin{lem}\label{lem:degree-bound}
  Let $\bk \in \N$. Let $q_{\bk} \in \C[t_1, \dots, t_p, y]$ be an
  irreducible polynomial such that
  $$
  q_{\bk}(t_1, \dots, t_p, \alpha_{\bk}) = 0.
  $$
  If $\alpha_{\bk} \neq 0$, then the degree of $q_{\bk}$ in each
  variable is at most $C_1 \cdot |\bk|$ for a suitable
  $C_1 = C_1(d, l)$.
\end{lem}
\begin{proof}
  This is a classical result, although usually only stated for the
  case $p = 1$, and either with $l - p = 1$, or with additional
  assumptions on the derivative in $y$ of the polynomial $f$. For the
  sake of completeness, we sketch an argument that reduces the general
  case to $p = 1$, $l - p = 1$.

  Suppose $p = 1$ and $l - p = 1$. In this case, the vector $\bk$ is
  just a single natural number $k \in \N$ and the degree of $q_k$ in
  $t_1$ coincides with the logarithmic height of $\alpha_k$ in the
  function field $F_{\infty}/\C$ (upon choosing an appropriate
  normalization). Then the height of $\alpha_k$ is at most $C \cdot k$
  for a suitable $C = C(d)$; moreover, there is a finite set of places
  $S$, whose size can be bounded in terms of $d$ only, such that all
  the coefficients $\alpha_k$ are $S$-integral. (See for instance
  \cite[Lem.\ V.5]{Mason1984}, with the additional observation that
  the order of $\frac{\partial f}{\partial y}$ at $x = 0$ can also be
  bounded in terms of $d$ only.)

  The general case can be reduced to the above special case by
  specializing $t_i \mapsto \xi_i t_1$ for $i = 1, \dots, p$, and
  $t_j \mapsto \chi_j t_l$ for $j = p+1, \dots, l$, where
  $\xi_i, \chi_j \in \C$ are algebraically independent over the field
  of definition of $f$. If $\beta_{\bk}$ is the specialization of
  $\alpha_{\bk}$, the coefficient of $t_l^k$ in the specialized
  expansion is
  $$
  \gamma_k = \sum_{|\bk| = k} \beta_{\bk} {\bm \chi}^{\bk},
  $$
  where ${\bm \chi} = (\chi_{p+1}, \dots, \chi_l)$. By the previous
  argument, the height of $\gamma_k$ is bounded by $C \cdot k$, and
  $\gamma_k$ is $S$-integral for a suitable $S$ of size bounded in
  terms of $d$ only. Since $S$ is contained among the poles of the
  functions $\beta_{\bk}$, upon varying $\bm \chi$ one recovers that
  $S$ does not depend on $\bm \chi$. Using sufficiently many
  independent values of $\bm \chi$, one can eliminate $\bm \chi$ and
  obtain that the height of each $\beta_{\bk}$ is bounded by
  $C \cdot |S| \cdot |\bk|$. It now suffices to note that
  $$
  r_{\bk}(t_1, y) := q_{\bk}(t_1, \xi_2 t_1, \dots, \xi_p t_1, y)
  $$
  can be factored as $r_{\bk} = t_{1}^e r_{\bk}'$ for some $e \leq dl$
  and some irreducible $r_{\bk}' \in \C[t_1, y]$, and that
  $r_{\bk}'(t_1, \beta_{\bk}) = 0$, so
  $\deg_{t_1}(r_{\bk}') \leq C \cdot |S| \cdot |\bk|$. Since the
  degree of $q_{\bk}$ in each $t_i$ is at most
  $\deg_{t_1}(r_{\bk}') + e$, the conclusion follows at once.
\end{proof}

\begin{prop}\label{prop:prim-degree-bound}
  For all $L \in \N$, there is a primitive element $\alpha$ of
  $F_L/\C(t_{1}, \dots, t_{p})$ such that, if
  $q \in \C[t_1, \dots, t_p, y]$ is an irreducible polynomial such
  that
  $$
  q(t_1, \dots, t_p, \alpha) = 0,
  $$
  then the degree of $q$ in each variable is at most
  $C_2 = C_2(d,l) \cdot L$. Moreover, we may assume that $q$ is monic
  in $y$.
\end{prop}
\begin{proof}
  By a classical argument of Galois theory, for sufficiently generic
  coefficients $\lambda_{\bk} \in \C$, the element
  $$
  \alpha = \sum_{|\bk| \leq L} \lambda_{\bk}\alpha_{\bk}
  $$
  generates $F_L$ over $\C(t_1, \dots, t_p)$. The conclusion then
  follows by \prettyref{lem:degree-bound} and some elementary algebra.
\end{proof}

\begin{prop}\label{prop:height-bound}
  For all $\bk \in \N^{l-p}$, if $\tilde{\alpha}_{\bk} \neq 0$, then
  either
  $h(\tilde{\alpha}_{\bk}) \leq C_1d^2(n_1 + \dots + n_p) \cdot
  |\bk|$, or there are integers $h_{1}, \dots, h_{p} \in \Z$, not all
  zero, such that $|h_{i}| \leq 2C_{1} \cdot |\bk|$ and
  $h_{1}n_{1} + \dots + h_{p}n_{p} = 0$.
\end{prop}
\begin{proof}
  By \prettyref{lem:degree-bound}, we have
  $$
  q_{\bk}(x^{n_1}, \dots, x^{n_p}, \tilde{\alpha}_{\bk}) =
  \tilde{q}_{\bk}(x, \tilde{\alpha}_{\bk}) = 0
  $$
  where $q_{\bk}$ has degree at most $C_1 \cdot |\bk|$ in each
  variable. Therefore, the specialized polynomial $\tilde{q}_{\bk}$
  has degree at most $C_{1}(n_{1} + \dots + n_{p}) \cdot |\bk|$ in
  $x$. If $\tilde{q}_{\bk} \neq 0$, then $h(\tilde{\alpha}_{\bk})$ is
  bounded by the height of the polynomial
  $q_{\bk}(x^{n_1}, \dots, x^{n_p}, y)$, and the first conclusion
  follows by an easy estimate of such height.

  Otherwise, if $\tilde{q}_{\bk}$ vanishes, then (at least) two
  distinct terms of $q_{\bk}$ become terms of the same degree in $x$
  when specialized. This immediately implies the second conclusion.
\end{proof}

\begin{cor}\label{cor:genus-bound}
  For all $L \in \N$, either the genus of $\tilde{F}_L$ is bounded by
  $C_3(n_1 + \dots + n_p) \cdot L$ for some $C_3 = C_3(d,l)$, or there
  are integers $h_{1}, \dots, h_{p} \in \Z$, not all zero, such that
  $|h_{i}| \leq 2C_{1} \cdot |\bk|$ and
  $h_{1}n_{1} + \dots + h_{p}n_{p} = 0$.
\end{cor}
\begin{proof}
  This follows at once from the estimate of
  \prettyref{prop:height-bound} and the basic theory of function
  fields (for instance, by counting the ramification points).
\end{proof}

\begin{rem}
  Recall that $\tilde{F}_L = \tilde{F}_{\infty}$ for all sufficiently
  large integers $L$. One can prove that this happens for all integers
  $L$ larger than a number dependent on $d$ only, showing that the
  above bound can be made independent of $L$ (as for
  \prettyref{lem:degree-bound}, this is usually proven only with
  additional assumptions on $f$). However, we will not need this
  additional uniformity.
\end{rem}

\section{Diophantine approximation}
\label{sec:approx}

Now that we have found a suitable expansion of $g(x)$ as a convergent
sum of algebraic functions, we proceed as in
\cite{Zannier2008}. Recall the following lemma.

\begin{lem}[\protect{\cite[Prop.\ 1]{Zannier2008}}]
  \label{lem:s-units}Let $E/\C$ be a function field in one variable,
  of genus $\mathfrak{g}$, $\varphi_{1},\ldots,\varphi_{n}\in E$ be
  linearly independent over $\C$ and $r \in \{0, 1, \dots, n\}$. Let
  $S$ be a finite set of places of $E$ containing all the poles of
  $\varphi_{1},\ldots,\varphi_{n}$ and also all the zeros of
  $\varphi_{1},\ldots,\varphi_{r}$. Further, put
  $\sigma=\sum_{i=1}^{n}\varphi_{i}$. Then
  \begin{equation}
    \sum_{v\in S}(v(\sigma)-\min_{i=1}^{n}v(\varphi_{i}))\le{n \choose 2}(\#S+2\mathfrak{g}-2)+\sum_{i=r+1}^{n}\deg(\varphi_{i}),\label{eq:s-units}
  \end{equation}
  where $\deg(\varphi_i) = [E:\C(\varphi_i)]$.
\end{lem}

A rather straightforward application of the above lemma to
(\ref{eq:pseudo-analytic}) yields the following (using the notations
of \prettyref{sec:multivariate}).

\begin{prop}\label{prop:applied-s-units}
  Suppose that $0 < n_1 \leq \dots \leq n_l$ and that
  $n_{p+1} \geq \varepsilon n_l$ for some given $\varepsilon >
  0$. Then at least one of the following holds:
  \begin{enumerate}
  \item $g(x)$ is $\C$-linearly dependent on the set
    $\{ \tilde{\alpha}_{\bk} \,:\, |\bk| \leq \lceil\frac{2d +
      1}{\varepsilon}\rceil\}$,
  \item $n_p \geq \varepsilon' n_l$ for some
    $\varepsilon' = \varepsilon'(d, l, \varepsilon)$,
  \item there are $h_{1}, \dots, h_{p} \in \Z$, not all zero, such
    that $|h_{i}| \leq 2C_{1}L$ and
    $h_{1}n_{1} + \dots + h_{p}n_{p} = 0$.
  \end{enumerate}
\end{prop}
\begin{proof}
  We apply \prettyref{lem:s-units} with the following data.

  Let $L := \lceil\frac{2d + 1}{\varepsilon}\rceil$. Let
  $\varphi_1, \dots, \varphi_r$ be a $\C$-linear basis of the set
  $\{ \tilde{\alpha}_{\bk}x^{\bk \cdot \bn} \,:\, |\bk| \leq L \}$,
  and let $\varphi_{r+1} = \varphi_n = -g(x)$ (note that $r$ is at
  most the number of vectors $\bk \in \N^{l-p}$ such that
  $|\bk| \leq L$, which can be bounded in terms of $l$ and $L$
  only). Let $E := \tilde{F}_L$. Let $S$ be the set of zeroes and
  poles of the functions $\tilde{\alpha}_{\bk}x^{\bk \cdot \bn}$ and
  of the function $x$. Let $v_{0}$ be the place at $x = 0$ induced by
  the embedding of $E$ into $\C[[x]]$. Thanks to the observations of
  \prettyref{sec:multivariate}, we either get conclusion (3), or we
  have the following bounds:
  \begin{itemize}
  \item the zeroes and poles of each $\tilde{\alpha}_{\bk}$ are at
    most $2C_1d^2pn_pL$ by \prettyref{prop:height-bound};
  \item the zeroes and poles of $x^{\bk \cdot \bn}$, which include the
    poles of $g(x)$, are at most $2[F_L:\C(x)] \leq 2d$ by
    \prettyref{prop:degree-coeffs};
  \item the genus of $F_L$ is at most $C_3pn_pL$ by
    \prettyref{cor:genus-bound};
  \item $v_0(\sigma) > Ln_{p+1} \geq L\varepsilon n_l$ by the
    expansion (\ref{eq:pseudo-analytic}).
  \end{itemize}

  Finally, note that for all $v \in S$,
  $v(\sigma) - \min_{i = 1}^n v(\varphi_i)$ is non-negative, and that
  by simple degree considerations,
  $0 \leq v_0(g(x)), \deg(g(x)) \leq dn_l$. In particular,
  $\min_{i = 1}^n(v_0(\varphi_i)) \leq dn_l$.

  Applying \prettyref{lem:s-units} and using these bounds, we either
  reach conclusion (1), or we obtain the following inequality:
  $$
  L\varepsilon n_l - dn_l \leq \binom{n}{2} (2C_1d^2pn_pL +
  2d + 2C_3pn_pL) + dn_l,
  $$
  which in turn yields
  $$
  n_l \leq n_l \cdot \left(L\varepsilon - 2d\right) \leq \binom{n}{2}
  (2C_1d^2pL + 2d + 2C_3pL) \cdot n_p,
  $$
  proving conclusion (2).
\end{proof}

\section{The case of linear dependence}
\label{sec:lin-dep}

Note that outcome (1) of \prettyref{prop:applied-s-units} is that
$g(x)$ is $\C$-linearly dependent on
$\{\tilde{\alpha}_{\bk} \,:\, |\bk| \leq L\}$ for a certain
$L \in \N$. In this section, we study what happens when this is the
case.

Let $L \in \N$. Fix $\alpha$ to be a primitive element of
$F_L/\C(t_1, \dots, t_p)$ given by \prettyref{prop:prim-degree-bound},
with the corresponding irreducible polynomial
$q \in \C[t_1, \dots, t_p, y]$. Let also
$e = [F_L:\C(t_1, \dots, t_p)]$. Then for all $|\bk| \leq L$ we can
write
$$
\alpha_{\bk} = \sum_{i = 0}^{e-1} q_{i,\bk}\alpha^i
$$
where $q_{i,\bk} \in \C(t_1, \dots, t_p)$.

\begin{lem}
  The degree of $q_{i,\bk}$ in each variable is at most $C_4$ for some
  $C_4 = C_4(d,l, L)$.
\end{lem}
\begin{proof}
  Let $\sigma : F_L \to \overline{\C(t_1, \dots, t_p)}$ be an
  embedding of $F_L$ into an algebraic closure of
  $\C(t_1, \dots, t_p)$. Then
  $$
  \sigma(\alpha_{\bk}) = \sum_{i = 0}^{e-1} q_{i,\bk}
  \sigma(\alpha)^i.
  $$
  Since the matrix $(\sigma(\alpha)^i)_{i,\sigma}$ is invertible, the
  desired bound follows from the bound of
  \prettyref{prop:prim-degree-bound} and some elementary algebra.
\end{proof}

\begin{prop}\label{prop:gx-lin-dep}
  Suppose that $g(x)$ is $\C$-linearly dependent on
  $\{ \tilde{\alpha}_{\bk}x^{\bk \cdot \bn} \,:\, |\bk| \leq
  L\}$. Assume that \prettyref{prop:main-regular} is true for $l =
  p$. Then either $h_1n_1 + \dots + h_pn_p = 0$ for some integers
  $h_1, \dots, h_p$, not all zero, such that $|h_i| \leq 2C_4$, or
  $g(x)$ can be written as the ratio of two polynomials with at most
  $C_5 = C_5(d,l,L)$ terms.
\end{prop}
\begin{proof}
  Assume first that for some $\bk$ with $|\bk| \leq L$,
  $q_{i,\bk}(x^{n_1}, \dots, x^{n_p})$ is not well-defined, by which
  we mean that some denominator of $q_{i,\bk}$ vanishes on
  $(x^{n_1}, \dots, x^{n_p})$. In turn, two distinct terms of such
  denominator must have the same degree when specialized, which means
  that
  $$
  h_1n_1 + \dots + h_pn_p = h_1'n_1 + \dots + h_p'n_p
  $$
  for some integers $h_i,h_i'$ such that $|h_i|, |h_i'| \leq C_4$ for
  all $i$, and $h_i \neq h_i'$ for at least one $i$. We thus reach the
  former conclusion.

  Otherwise, we specialize $\alpha$ at $t_i = x^{n_i}$ for
  $i = 1, \dots, p$, which is possible since $\alpha$ is by
  construction integral over $\C[t_1, \dots, t_p]$. This yields a
  $\tilde{\alpha} \in \tilde{F}_L$. By the above argument, we can also
  specialize each $q_{i,\bk}$, yielding the following:
  $$
  \tilde{\alpha}_{\bk} = \sum_{i = 0}^{e-1} q_{i,\bk}(x^{n_1}, \dots,
  x^{n_p}){\tilde{\alpha}}^i.
  $$
  In particular, $\tilde{\alpha}$ is a primitive element of
  $\tilde{F}_L$ over $\C(x)$.

  By assumption of linear dependence, we have
  $$
  g(x) = \sum_{|\bk| \leq L} \lambda_{\bk}\tilde{\alpha}_{\bk}x^{\bk
    \cdot \bn} = \sum_{|\bk| \leq L} \lambda_{\bk}x^{\bk \cdot \bn}
  \sum_{i = 0}^{e-1}q_{i,\bk}\tilde{\alpha}^i
  $$
  for some numbers $\lambda_{\bk} \in \C$.

  Let $[\C(x,\tilde{\alpha}):\C(x)] = e' \leq e$. If $e' = e$, then
  $\tilde{\alpha}^0, \dots, \tilde{\alpha}^{e-1}$ are $\C(x)$-linearly
  independent, so we must have
  $$
  g(x) = \sum_{|\bk| \leq L} \lambda_{\bk}x^{\bk \cdot \bn}q_{0,\bk},
  $$
  and the latter conclusion follows.

  If $\tilde{\alpha}^0, \dots, \tilde{\alpha}^{e-1}$ are not
  $\C(x)$-linearly independent, it means that
  $[\C(x,\tilde{\alpha}):\C(x)] = e' < e$. Note that $\tilde{\alpha}$
  is a root of
  $$
  q(x^{n_1}, \dots, x^{n_p}, \tilde{\alpha}) = 0.
  $$
  By hypothesis, we may assume that \prettyref{prop:main-regular}
  holds for $l = p$, and in particular we may apply
  \prettyref{thm:main-factor}. It follows that the irreducible factor
  of $q(x^{n_1}, \dots, x^{n_p}, y)$ of which $\tilde{\alpha}$ is a
  root has coefficients that can be written as ratio of polynomials
  with at most $B_5(d,p)$ terms. In turn, we may rewrite each power
  $\tilde{\alpha}^j$, for $j \geq e'$, as
  $$
  \tilde{\alpha}^j = \sum_{i = 0}^{e'-1}r_{i,j}\tilde{\alpha}^i
  $$
  where each $r_{i,j} \in \C(x)$ can be written as the ratio of two
  polynomials whose number of terms is bounded in terms of $d$ and $p$
  only.

  Therefore, we have
  $$
  g(x) = \sum_{|\bk| \leq L} \lambda_{\bk}x^{\bk \cdot \bn} \sum_{i =
    0}^{e'-1}\left(q_{i,\bk} + \sum_{j = e'}^e
    q_{j,\bk}r_{i,j}\right)\tilde{\alpha}^i.
  $$
  Since $\tilde{\alpha}^0, \dots, \tilde{\alpha}^{e'-1}$ are
  $\C(x)$-linearly independent, we may conclude as in the case
  $e' = e$.
\end{proof}

\section{Proof of the main theorem}
\label{sec:induction}

Finally, we can prove \prettyref{prop:main-regular}. We shall then
prove that it implies \prettyref{thm:main-poly}.

\begin{proof}[Proof of \prettyref{prop:main-regular}]
  First of all, we may directly assume that
  $0 < n_1 \leq \dots \leq n_l$: indeed, we may simply rearrange
  $t_1, \dots, t_l$ as required.

  We then work by primary induction on $l \in \N^{*}$ and secondary
  reverse induction on $p = l - 1, \dots, 0$. Our inductive hypothesis
  at stage $(l, p)$ reads as follows:
  \begin{enumerate}
  \item either $g(x)$ is the ratio of two polynomials with at most
    $B_9(d,l,p)$ terms,
  \item or $n_p \geq \varepsilon_pn_l$ for some
    $\varepsilon_p = \varepsilon_p(d, l)$.
  \end{enumerate}

  Note that in the case $(l, p) = (1, 0)$, the conclusion is trivial:
  the expansion (\ref{eq:pseudo-analytic}) is just
  $$
  g(x) = \sum_{k \in \N} \tilde{\alpha}_k x^{kn_l}
  $$
  where $\tilde{\alpha}_k \in \C$. Since the degree of $g$ in $x$ is
  at most $dn_l$, by simple degree considerations, $g(x)$ is a
  polynomial with at most $d+1$ terms, reaching conclusion (1).

  Similarly, in the case $(l, 0)$, the expansion is of the type
  $$
  g(x) = \sum_{\bk \in \N^l} \tilde{\alpha}_{\bk} x^{\bk \cdot \bn}
  $$
  where $\tilde{\alpha}_{\bk} \in \C$. Since again the degree of $g$
  in $x$ is at most $dn_l$, the only terms appearing on the right hand
  side satisfy $\bk \cdot \bn \leq dn_l$. On the other hand,
  $\bk \cdot \bn \geq |\bk|n_1 \geq |\bk|\varepsilon_0n_l$. Therefore,
  $$
  |\bk| \leq \frac{d}{\varepsilon_0}.
  $$
  In turn, this implies that the number of terms of $g(x)$ can be
  bounded in terms of $l$ and $\varepsilon_0 = \varepsilon_0(d,l)$, so
  in terms of $d$ and $l$ only, reaching again conclusion (1).

  We now wish to prove the general case for arbitrary $(l, p)$. Assume
  that either $p = l-1$, or that we have already proven stage
  $(l, p+1)$. \prettyref{prop:applied-s-units} yields three possible
  conclusions: in the first case, we obtain that $g(x)$ is
  $\C$-linearly dependent on
  $\{ \tilde{\alpha}_{\bk}x^{\bk \cdot \bn} \,:\, |\bk| \leq L \}$ for
  a suitably chosen $L$; in the second case, we reach conclusion (2)
  straightaway; in the third case, we obtain that
  $h_1n_1 + \dots + h_pn_p = 0$ for some integers $h_i$ of bounded
  size, in which case we reach conclusion (1) by
  \prettyref{lem:linear-dep-exp} and the inductive hypothesis.

  For the first case, we then apply \prettyref{prop:gx-lin-dep}, and
  we either reach conclusion (1) immediately, or we find again that
  $h_1n_1 + \dots + h_pn_p = 0$ for integers $h_i$ of bounded size, so
  we reach conclusion (1) again by \prettyref{lem:linear-dep-exp} and
  the inductive hypothesis. This concludes our induction.
\end{proof}

Chasing back the series of deductions, this finally proves
\prettyref{thm:main}. The proof of \prettyref{thm:main-poly} now
follows the same argument found in \cite{Zannier2008}.

\begin{proof}[Proof of \prettyref{thm:main-poly}]
  By \prettyref{thm:main}, a rational function $g(x)\in\C(x)$ such
  that
  \[
    f(x^{n_{1}},\dots,x^{n_{l}},g(x))=0
  \]
  can always be written as the ratio of two polynomials, say
  $g_{1}(x)$ and $g_{2}(x)$, with at most $B_{4}$ terms.

  As in \cite{Zannier2008}, we may exploit this information to show
  that we may explicitly parametrize all such polynomials $g_{1}$,
  $g_{2}$. Indeed, for $\bk \in \N^l$, let $|\bk|_{\infty}$ be the
  maximum absolute value of its entries; for $r=1,2$, if
  \[
    g_{r}(x)=\sum_{k=1}^{B_{4}}b_{rk}x^{n_{rk}}
  \]
  and, writing ${\bf t}^{\bk}$ for $t_1^{k_1}\cdot\ldots\cdot t_l^{k_l}$,
  \[
    f(t_{1},\dots,t_{l},y)=\sum_{i=0}^{d}\sum_{|\bk|_{\infty}\leq
      d}a_{i\bk}y^{i}{\bf t}^{\bk},
  \]
  we have that
  \begin{equation}
    \sum_{i=0}^{d}\sum_{|\bk|_{\infty}\leq d}a_{i\bk}\left(\sum_{k=1}^{B_{4}}b_{1k}x^{n_{1k}}\right)^{i}\left(\sum_{k=1}^{B_{4}}b_{2k}x^{n_{2k}}\right)^{d-i}x^{\bn\cdot\bk}=0.\label{eq:two-cond}
  \end{equation}

  We now expand all the involved products to get monomials of the
  shape $\gamma x^{\mu}$, where $\gamma$ is a monomial in the
  coefficients $a_{i\bk}$ and $b_{rk}$, and $\mu$ is a positive
  $\Z$-linear combination of the exponents $n_{rk}$ and $n_{i}$. In
  order to satisfy \prettyref{eq:two-cond}, we can recognize two types
  of conditions.

  \medskip

  (I) The first type concerns the exponents $\mu$ of $x$.  We can
  partition the monomials $\gamma x^\mu$ by grouping the ones with the
  same $\mu$. For each set of the partition, the corresponding
  expressions of $\mu$ must have the same value, producing several
  vanishing homogeneous linear forms with integer coefficients in the
  $n_i,n_{rk}$. Note that the coefficients of such linear forms are
  bounded in terms of $d$ only. Moreover, since the number of possible
  partitions is bounded in terms of $d$ and $l$, there is a bound on
  the number of resulting linear equations.

  \medskip

  (II) For a fixed partition of the monomials $\gamma x^\mu$ with the
  same $\mu$ as in (I), the sum of their coefficients must be
  zero. This yields an affine algebraic variety whose coordinates
  correspond to the coefficients $b_{rk}$.

  \medskip

  Each solution $g_1(x),g_2(x)$ of \prettyref{eq:two-cond} yields a
  solution to a linear equation as in (I) and a point on the
  corresponding algebraic variety given in (II). Vice versa, each
  solution to a linear equation as in (I) and a point on the
  corresponding algebraic variety in (II) yield two polynomials
  $g_1(x),g_2(x)$ satisfying \prettyref{eq:two-cond}.

  Suppose now that we fix a set of linear equations as in (I), given
  by a partition of the exponents, and a point in the algebraic
  variety found in (II), but we let the exponents $n_{rk}$ vary among
  all the possible solutions. Since the (vector) solutions of such a
  system of linear equations span a subgroup of $\Z^{2B_{4}}$, we may
  in fact find a $\Z$-basis, say with $s\leq2B_{4}$ elements, whose
  entries are bounded only in terms of $d$ and $l$; we may then write
  each solution as linear combinations of these basis vectors, with
  integer coefficients $u_{1},\dots,u_{s}$.  After this substitution,
  we may rewrite the resulting polynomials $g_{1}$ and $g_{2}$ as
  \[
    g_{r}(x)=\tilde{g}_{r}(x^{u_{1}},\dots,x^{u_{s}}),\; r=1,2
  \]
  and $f$ as
  \[
    f(x^{n_{1}},\dots,x^{n_{l}},y)=\tilde{f}(x^{u_{1}},\dots,x^{u_{s}},y),
  \]
  where $\tilde{f}$, $\tilde{g}_{1}$ and $\tilde{g}_{2}$ are certain
  Laurent polynomials in $\C[z_{1}^{\pm1},\dots,z_{s}^{\pm1},y]$. Note
  that moreover their degrees are bounded in terms of the basis
  vectors and hence may be bounded in terms of $d$ and $l$ only.

  Now, the equality
  \[
    \tilde{f}\left(x^{u_{1}},\dots,x^{u_{s}},\frac{\tilde{g}_{1}(x^{u_{1}},\dots,x^{u_{s}})}{\tilde{g}_{2}(x^{u_{1}},\dots,x^{u_{s}})}\right)=0
  \]
  is satisfied for \emph{all} $u_{1},\dots,u_{s}$ in $\Z$, and
  therefore we actually have that
  \[
    \tilde{f}\left(z_{1},\dots,z_{s},\frac{\tilde{g}_{1}(z_{1},\dots,z_{s})}{\tilde{g}_{2}(z_{1},\dots,z_{s})}\right)=0.
  \]

  Since $\tilde{f}$ is monic in $y$, this implies that
  $\frac{\tilde{g}_{1}}{\tilde{g}_{2}}$ is integral over
  $\C[z_{1}^{\pm1},\dots,z_{s}^{\pm1}],$ and therefore it is a Laurent
  polynomial in $\C[z_{1}^{\pm1},\dots,z_{s}^{\pm1}]$; moreover, the
  number of terms as a Laurent polynomial is bounded dependently on
  $d$ and $l$ because the degree of $\tilde{f}$ is likewise bounded.

  Therefore, since any $g(x)$ satisfying \prettyref{eq:main-poly} can
  be obtained using the above procedure, we have that $g(x)$ must be a
  Laurent polynomial in $\C[x^{\pm1}]$ with a number of terms bounded
  dependently on $d$ and $l$. Now, since $g(x)$ is integral over
  $\C[x]$, then all of its monomials have non-negative degree, and
  therefore it is a polynomial with a bounded number of terms, as
  desired.
\end{proof}

\section{Proofs of the remaining assertions}

From \prettyref{thm:main} we can now deduce the various statements
given in \prettyref{sec:intro} with relatively small effort.

\medskip

We first prove \prettyref{thm:vojta} and its \prettyref{cor:vojta} on
integral points, i.e., regarding the regular maps $\rho:\Gm\to W$ for
a given finite cover $W\to \Gm^l$.

\begin{proof}[Proof of \prettyref{thm:vojta}]
  We first note that it suffices to prove the conclusion for a finite
  set of regular functions $y$ on $W$. Therefore, we may assume that
  $W$ may be represented as the hypersurface $f(t_1,\dots,t_l,y)=0$,
  where $f$ is monic in $y$, Laurent in the $t_i$'s, and $\pi$ is the
  projection onto the first $l$ coordinates.

  With this proviso, we go to the proof.  A regular map
  $\rho:\Gm\to W$ may be represented in the form
  $x\mapsto(\theta_{1}x^{m_{1}},\ldots,\theta_{r}x^{m_{r}},g(x))$,
  where $\theta_{i}\in\C^{*}$, $m_{i}\in\Z$ and $p\in\C[x,x^{-1}]$.
  Thus $f(\theta_{1}x^{m_{1}},\ldots,\theta_{r}x^{m_{r}},g(x))=0$.

  By \prettyref{thm:main}, and using the same argument of the proof of
  \prettyref{thm:main-poly}, we see that each choice of the
  coefficients $\theta_i$ and of the polynomial $g(x)$ corresponds to
  an integer solution of a system of linear equations (I) and to a
  point on an algebraic variety (II).

  Now, for each system (I), let $s$ be the rank of its solution space,
  and let $V$ be the corresponding algebraic variety (II). By
  construction, we obtain a map $\psi : V \times \Gm^s \to W$. The
  above comment on $\theta_i$ and $g(x)$ implies that there is map
  $\gamma : \Gm\to \Gm^s$, given by the solution of the system (I)
  corresponding to $g(x)$, and a point $\xi \in V$ corresponding to
  the coefficients of $g(x)$ and the $\theta_i$'s, such that in fact
  $\rho = \psi_\xi \circ \gamma$. Since the number of possible
  systems, and therefore of maps $\psi$, is bounded in terms of $d$
  and $l$, this yields the desired conclusion.
\end{proof}

\begin{proof}[Proof of \prettyref{cor:vojta}]
  We first remark a few things about the conclusion of
  \prettyref{thm:vojta}. First, we observe that since a regular map
  from $\Gm^{s}$ to $\Gm$ is a monomial, each
  $\pi\circ\psi:V\times\Gm^{s}\to\Gm^{l}$ is of the shape
  $\{\xi\}\times(z_{1},\ldots,z_{s})\mapsto(c_{1}(\xi)\mu_{1},\ldots,c_{l}(\xi)\mu_{l})$,
  for non vanishing functions $c_{i}$ on $V$ and pure monomials
  $\mu_{i}$ in the $z_{j}$. Also, since the map
  $(z_{1},\ldots,z_{s})\mapsto(\mu_{1},\ldots,\mu_{l})$ is a
  homomorphism, after an automorphism of $\Gm^{s}$ it factors as a
  projection $\Gm^{s}\to\Gm^{t}$ times a homomorphism with finite
  kernel; hence, $t\le l$, and we may in fact take $s=t\le l$.
  (Indeed, the map $\psi:V\times\Gm^{t}\times\Gm^{s-t} \to W$ sends
  $\{\xi\}\times\{\eta\}\times\Gm^{s-t}$ to a fiber of $\pi$, which is
  finite; hence this image is constant, and we may remove $\Gm^{s-t}$
  from the picture.)

  Then, after pullback of $\pi$ by an isogeny, we may assume that
  $\Gm^{t}$ embeds in $\Gm^{l}$ on the first $t$
  coordinates. Therefore, we can see that the map $\psi$ yields a
  family of translates of $\Gm^{t}$ parametrized by $V$, and
  corresponding regular sections of $\pi$ over each of them.

  Turning back to the proof, we note that the hypothesis combined with
  \prettyref{thm:vojta} imply immediately that one of the maps
  $\psi\in\Psi$ is dominant. Therefore, the composition
  $\pi\circ\psi:V\times\Gm^s\to\Gm^l$ is regular, dominant and (by the
  previous remarks) we may even suppose that it is expressed in the
  shape
  $\pi\circ\psi(\{\xi\}\times(z_1,\dots,z_s)) =
  (c_1(\xi)z_1,\dots,c_s(\xi)z_s,c_{s+1}(\xi)\mu_{s+1},\dots,c_l(\xi)\mu_l)$
  where $\mu_i$ are monomials in $z_1,\dots,z_s$, $c_1,\dots,c_l$ are
  non-vanishing regular functions on $V$, and $s\geq 1$.

  If we fix a point $\xi\in V$, the restriction of $\pi\circ\psi$ to
  $\{\xi\}\times\Gm^s$ is an isogeny, and therefore unramified.  We
  define $R\subset W$ as the ramification divisor of $\pi$, and
  $S=\pi(R)\subset \Gm^l$ as the branch locus. Let, for
  $z\in \Gm^{l-s}$, $K_z:=\pi^{-1}(\Gm^s \times \{z\})$. Note that
  $K_z$ may be reducible, even for all $z$. However, the image of
  $\pi\circ \psi$ restricted to $ \{\xi\} \times \Gm^s $ is of the
  shape $\Gm^s \times \{\phi(\xi)\}$ (where $\phi$ is a certain
  regular map $\phi:V\to \Gm^{l-s}$), and the map is essentially an
  isogeny and is finite. Then we have that
  $\psi(V \times \Gm^s)\cap K_z$ consists of a finite union of
  components $C$ of $K_z$ such that $\pi(C)=\Gm^s \times \{z\}$.

  Since $\psi$ is dominant, it follows easily (by counting dimensions)
  that $\psi(V\times\Gm^s)$ can miss a whole component of $K_z$ only
  for $z$ in a proper closed subset $E$ of $\Gm^{l-s}$. On the other
  hand, since the said map is essentially an isogeny, $R$ cannot meet
  its image, so $R\cap K_z$ is contained in the components missed by
  $\psi(V\times\Gm^s)\cap K_z$.

  Therefore $R\cap K_z$ can be nonempty only for $z\in E$, and then
  the projection of $S$ to $\Gm^{l-s}$ is contained in $E$. Since $S$
  has pure codimension $1$ in $\Gm^l$, it follows that $S$ is a union
  of cosets of $\Gm^s$, and is therefore invariant by multiplication
  by $\Gm^s$.
\end{proof}

The proof of the toric version of Bertini's theorem \ref{thm:bertini}
follows a similar pattern:

\begin{proof}[Proof of \prettyref{thm:bertini}]
  Let us assume first that $W$ is representable as an open dense
  subset of the hypersurface $f(t_1, \dots, t_l, y) = 0$, where $f$ is
  an irreducible complex polynomial, and $\pi$ is the projection onto
  the first $l$ coordinates.

  Let us analyze a factorization
  $f(\theta_{1}x^{n_{1}},\ldots,\theta_{l}x^{n_{l}},y)=g(x,y)h(x,y)$
  with integers $n_{i}$ and polynomials (Laurent in $x$) $g,h$, monic
  in $y$. By \prettyref{thm:main}, and proceeding as in the proof of
  \prettyref{thm:main-poly}, we can see that the pairs $g,h$
  correspond to solutions of suitable systems (I) and to points on the
  corresponding affine algebraic varieties (II).

  Now, fix a system (I) and a point on the algebraic variety of
  (II). As before, if $s$ is the rank of the solution space, we can
  easily obtain the following factorization:
  \begin{equation}
    f(\theta_{1}\mu_{1}, \ldots, \theta_{l}\mu_{l}, y) = \tilde{g}(z_{1}, \ldots, z_{s}, y) \tilde{h}(z_{1}, \ldots, z_{s}, y), \label{eq:bertini}
  \end{equation}
  where $\mu_1,\dots,\mu_l$ are (Laurent) monomials in $z_1$, $\dots$,
  $z_s$ and $\tilde{g}$, $\tilde{h}$ are polynomials (Laurent in the
  $z_i$'s) and monic in $y$.

  Now, suppose the monomials $\mu_{1},\ldots,\mu_{l}$ are
  multiplicatively independent. This means that the homomorphism
  $\phi:\Gm^{s}\to\Gm^{l}$ given by
  $\phi(z_{1},\ldots,z_{s})=(\mu_{1},\ldots,\mu_{l})$ is surjective.
  By simple general theory, it must factor as a composition of a
  projection $\Gm^{r}\times\Gm^{s-r}\to\Gm^{r}$ and an isogeny $\psi$
  of $\Gm^{r}$. But then the identity (\ref{eq:bertini}) shows that
  the pullback $\psi^{*}W$ is reducible; now, it is known and not too
  difficult to prove that this implies that $[e]^{*}W$ is already
  reducible (see \cite{Zannier2010}, Prop.\ 2.1), against the
  assumptions.

  Therefore, we may assume that in all cases the sets of monomials
  $\mu_{i}$ so obtained are multiplicatively dependent, hence they
  satisfy an identical relation
  $\mu_{1}^{e_{1}}\cdots\mu_{r}^{e_{r}}=1$ for integer exponents
  $e_{i}$, not all zero and depending only on the linear form chosen
  in (I). In particular, the vector $(e_{1},\ldots,e_{r})$ takes
  altogether only finitely many values.

  Since the $\mu_{i}$'s are pure monomials in the $z_{h}$, we may
  assume that the $e_{i}$'s are coprime. The multiplicative relation
  defines a certain proper connected algebraic subgroup $E$ of
  $\Gm^{l}$, while the corresponding factorization implies that
  $\pi^{-1}(\theta E)$ is reducible for
  $\theta=(\theta_1,\dots,\theta_l)$. Therefore, the original
  $1$-dimensional torus parametrized by $(x^{n_{1}},\ldots,x^{n_l})$
  is contained in $E$. We now let $\mathcal{E}$ to be the union of all
  finitely many sub-tori $E$ which arise in this way. Note that
  $\mathcal{E}$ can be chosen dependently only on $\deg(f)$.

  Now, assume that $\pi^{-1}(\theta H)$ is reducible, for a certain
  $\theta\in\Gm^{l}$ and a certain torus $H$ of dimension $t\ge1$. If
  $(u_{1},\ldots,u_{t})\mapsto(\nu_{1},\ldots,\nu_{l})$ is a
  parametrization of $H$ by monomials $\nu_{i}$ in the $u_{h}$, then
  $f(\theta_{1}\nu_{1},\ldots,\theta_{r}\nu_{r},y) = 0$ is reducible
  (over $\C(u_{1},\ldots, u_{t})$). Hence, simply by specialization,
  the polynomial $f(\theta_{1}x^{n_{1}},\ldots,\theta_{l}x^{n_{l}},y)$
  must be reducible for all integer vectors $(n_{1},\ldots,n_{l})$
  such that the torus $(x^{n_{1}},\ldots,x^{n_{l}})$ is contained in
  $H$.  But then any such torus must be contained in some $E$ as
  above; it is now easy to see that $H$ itself must be contained in
  $\mathcal{E}$, proving the desired conclusion.

  \medskip

  To complete the proof, consider a general quasi-projective variety
  $W$. After replacing $W$ with $W \setminus X$ for a suitable proper
  subvariety $X$, we may assume that $\pi : W \to \Gm^l$ is finite
  onto its image. We note that we may cover $W$ with finitely many
  (open dense) affine charts, such that for any two points of $W$
  there is a chart containing both of them; since $\pi$ is finite over
  its image, we may further assume that each chart can be represented
  as an open dense subset of the hypersurface
  $f(t_1, \dots, t_l, y) = 0$ for some $f$. We then observe that if
  $\pi^{-1}(\theta H)$ has at least two irreducible components, for
  some subgroup $H < \Gm^l$ and some $\theta \in \Gm^l$, then there is
  at least one affine chart intersecting both components, and the
  conclusion follows by the previous case.
\end{proof}

Finally, the only remaining statement is the analogue for composite
rational functions of Schinzel's conjecture, namely
\prettyref{thm:ratfunc1}.

\begin{proof}[Proof of Theorem \ref{thm:ratfunc1}]
  Let $l$ be given and let $f(x)=g(h(x))$ be as in the statement. We
  write $f(x)=P(x)/Q(x)$ with
  $P(x)=p_1x^{n_1}+\cdots
  +p_lx^{n_l},Q(x)=q_1x^{n_1}+\cdots+q_lx^{n_l}, P(x),Q(x)\in\C[x]$.
  If we put $d=2016\cdot 5^l$, we know by the main theorem of
  \cite{Fuchs2012} that $\deg g\leq d$ unless we are in the
  exceptional situation of that theorem, where our statement is
  trivially true. Therefore we may write $g(x)=A(x)/B(x)$ with
  $A(x)=a_0+a_1x+\cdots +a_dx^d,B(x)=b_0+b_1x+\cdots+b_dx^d$ be two
  (coprime) polynomials in $\C[x]$. From $f(x)=g(h(x))$ we therefore
  get \[A(h(x))Q(x)-B(h(x))P(x)=0.\]

  We then define
  \[f(t_1,\ldots,t_l,y)=A(y)(q_1t_1+\cdots
  +q_lt_l)-B(y)(p_1t_1+\cdots+p_lt_l)\in\C[t_1,\ldots,t_l,y].\]
  This is a polynomial of degree at most $d$ in each variable. An
  application of Theorem \ref{thm:main} shows at once that there
  exists a number $B_{2}=B_{2}(l)=B_{4}(d,l)$ such that $h(x)\in\C(x)$,
  which satisfies $f(x^{n_1},\ldots,x^{n_l},h(x))=0$, is the ratio of
  two polynomials in $\C[x]$ with a most $B_{2}$ terms, as desired.
\end{proof}

\end{document}